\documentclass[reqno, a4paper]{amsart}
\usepackage{amssymb, mathdots}
\usepackage{mathabx}
\usepackage{mathrsfs}
\usepackage[utf8]{inputenc}
\usepackage{amsmath,  graphicx, tensor}
\usepackage{tikz}
\usetikzlibrary{matrix, arrows.meta}
\usetikzlibrary{cd}

\usepackage[all]{xy}
\usepackage{enumerate}
\usepackage[percent]{overpic}
\usepackage{hyperref}
\DeclareSymbolFont{bbold}{U}{bbold}{m}{n}
\DeclareSymbolFontAlphabet{\mathbbold}{bbold}
\def\qmod#1#2{{\hbox{}^{\displaystyle{#1}}}\!\big/\!\hbox{}_{
\displaystyle{#2}}}

 \def\psp#1#2%
  {\mathop{}%
   \mathopen{\vphantom{#2}}^{#1}%
   \kern-\scriptspace%
    \hskip -0.3mm{#2}} 
 \def\psb#1#2%
  {\mathop{}%
   \mathopen{\vphantom{#2}}_{#1}%
   \kern-\scriptspace%
    \hskip -0.0mm{#2}} 
 \def\pscr#1#2#3%
  {\mathop{}%
   \mathopen{\vphantom{#3}}^{#1}_{#2}%
   \kern-\scriptspace%
    \hskip -0.3mm{#3}} 

\def\C{{\mathbb C}}

\def\N{{\mathbb N}}

\def\P{{\mathbb P}}

\def\Z{{\mathbb Z}}

\def\union{\mathop{\bigcup}}

\def\textmap#1{\mathop{\vbox{\ialign{
                                  ##\crcr
      ${\scriptstyle\hfil\;\;#1\;\;\hfil}$\crcr
      \noalign{\kern 1pt\nointerlineskip}
      \rightarrowfill\crcr}}\;}}
\def\bigtextmap#1{\mathop{\vbox{\ialign{
                                  ##\crcr
      ${\hfil\;\;#1\;\;\hfil}$\crcr
      \noalign{\kern 1pt\nointerlineskip}
      \rightarrowfill\crcr}}\;}}
      
\newcommand{\cal}{\mathcal}
\def\textlmap#1{\mathop{\vbox{\ialign{
                                  ##\crcr
      ${\scriptstyle\hfil\;\;#1\;\;\hfil}$\crcr
      \noalign{\kern-1pt\nointerlineskip}
      \leftarrowfill\crcr}}\;}}
 
\def\ag{{\mathfrak a}}

\def\mg{{\mathfrak m}}

\def\Qg{{\mathfrak Q}}

\def\Ec{\mathcal{E}}

\def\Oc{\mathcal{O}}

\theoremstyle{remark}
\newtheorem{ex}{Example}[section]

\newtheorem{sz}{Satz}[section]
\theoremstyle{remark}

\theoremstyle{plain}
\newtheorem{re}[sz]{Remark} 
\newtheorem{thr}[sz]{Theorem}

\newtheorem{pr}[sz]{Proposition}
\newtheorem{co}[sz]{Corollary}
\newtheorem{dt}[sz]{Definition}
\newtheorem{lm}[sz]{Lemma}

\def\End{\mathrm {End}}
\def\Aut{\mathrm {Aut}}

\def\GL{\mathrm {GL}}

\def\Pic{\mathrm {Pic}}

\def\deg{\mathrm {deg}}

\def\id{ \mathrm{id}}
\def\im{\mathrm{im}}
\def\rk{\mathrm {rk}}

\def\ev{\mathrm{ev}}

\def\coker{\mathrm{coker}}

\def\Ext{\mathrm{Ext}}

\newcommand\smvee{{\hskip -0.1ex \raise 0.2ex\hbox{$\scriptscriptstyle\vee$}}\hskip -0,3ex}

\def\trp#1{\tensor[^{\mathrm{t}}]{#1}{}}
\def\edf{\coloneq}

\def\Coker{\mathrm{Coker}}

\def\Aff{{\rm{Aff}}}

\def\crk{\mathrm{crank}}

\def\Ann{\mathrm{Ann}}

\begin{document}

\title{Affine holomorphic bundles over $\P^1_\C$ and apolar ideals}

\author{Naoufal Bouchareb}
\address{Aix Marseille Univ, CNRS, I2M, Marseille, France.}
\email[Bouchareb Naoufal]{noufelbouchareb@hotmail.fr}

\begin{abstract}

We study the classification of affine holomorphic bundles over a compact complex manifold $X$ in general,  and we apply the general theory to the case  $X=\P^1_\C$. We study  the moduli space of framed, non-degenerate rank 2 affine bundles over $\P^1_\C$ whose linearisation, viewed as locally free sheaf,  is isomorphic to    $ {\cal O}_{\P^1_\C}(n_1)\oplus {\cal O}_{\P^1_\C}(n_2)$ where $n_1>n_2$. We show that this moduli space can be identified with the ``topological cokernel" of a morphism of linear spaces  over the projective space $\P(\C[X_0,X_1]_{l})$ of binary forms of degree $l\edf -2-n_2$, in particular it fibres over this projective space with vector spaces as fibres. We show that the stratification of $\P(\C[X_0,X_1]_{l})$ defined by the level sets of the fibre dimension map is determined explicitly by $d\edf n_1-n_2$ and the cactus rank stratification of $\P(\C[X_0,X_1]_{l})$.     
 \end{abstract}


\maketitle

\tableofcontents

\section*{Acknowledgements}

I am indebted to my PhD advisor, Andrei Teleman for suggesting me the problems studied in this article and many useful ideas. I am also grateful to Alexandru Dimca for pointing me out the relevance of the theory of apolar ideals, the role of the cactus rank, and for directing me to the very useful articles  \cite{DS1}, \cite{DS2}. I thank professors Karl Oeljeklaus, Anne Pichon and Frédéric Mangolte for their interest in my work and for their encouragement and support.

\section{Introduction}\label{intro}

\subsection{Previous work on holomorphic affine bundles. Our goals}

In his famous work \cite{We}, Weil introduces and studies affine line bundles over an algebraic curve $C$. Using the terminology of this article (see section \ref{prelim}) the objects studied by Weil are algebraic affine line bundles over    $C$ whose linearisation is isomorphic to the line bundle associated with a divisor $D$ on $C$.  His classification result \cite[section 4]{We} can be recovered as a special case  of (the algebraic-geometric analogue of) Theorem \ref{ClassTh-intro} in this article, taking  into account Serre duality.  Since Weil's method does not generalise to arbitrary dimension,  the proof of our Theorem \ref{ClassTh-intro}  uses different ideas and techniques.
\vspace{2mm}

Affine line bundles in the holomorphic category have been considered in Enoki's article about class VII surfaces with curves \cite{En}. In this article the author proves that any minimal class VII surface with $b_2>0$ which contains a non-zero homologically trivial  divisor can be obtained as a compactification of the total space of an  affine line bundle with negative degree over an elliptic curve. The compactifying divisor is a homologically trivial cycle of rational curves.
\vspace{2mm}

The classification of affine line bundles in the complex geometric category has been studied by Plechinger \cite{Pl1}, \cite{Pl2}. His main results are:

\begin{itemize} 
\item A generalisation (with different methods)  of Weil theorem for arbitrary complex manifolds. 
\item Classification results for framed affine line bundles on compact complex manifolds. In Plechniger's formalism, an $x_0$-framed affine line bundle on $X$ is a pair $(A,\varepsilon)$, where $A$ is an affine   on $X$ and $\varepsilon$ is  a linear isomorphism between  $\C$ and the fibre $L(A)_{x_0}$ of the linearisation $L(A)$ of $A$ at a fixed point $x_0\in A$. He gives a sufficient and necessary condition for the existence of a moduli space in the sense of classical complex geometry  which classifies $x_0$-framed affine line bundles $A$ on $X$  with fixed Chern class $c_1(L(A))=c\in H^2(X,\Z)$,   up to isomorphism.    If this existence condition is satisfied,  the obtained moduli space is a linear space over the component $\Pic^c(X)$ of the Picard group $\Pic(X)$.  The role of the framing is clear: the moduli space of non-framed affine line bundles with $c_1(L(A))=c$ can be recovered as the (highly non-Hausdorff) $\C^*$-quotient of the moduli space of framed line bundles.

\item For the general case, Plechinger constructed and studied  the holomorphic stack which classifies framed  and non-framed holomorphic affine line bundles  which replaces (in the framed case) the classical moduli space when the existence condition is not satisfied. 
  
\end{itemize}

Our main results (see section \ref{prelim} for details) concern the classification of higher rank  holomorphic  affine bundles over compact complex manifolds. We will start with a general classification theorem for affine bundles whose linearisation belongs to a fixed isomorphism class (Theorem \ref{ClassTh-intro}). 
We will illustrate this general theorem in the special case $X=\P^1_\C$ (see Proposition \ref{general-case-prop}), which shows already that, even on the simplest Riemann surface, the classification of affine bundles is much more complicated than the classification of vector bundles.

Afterwards we will focus on affine rank 2-bundles on a Riemann surface $C$, more precisely on such affine bundles on $C$ with {\it unstable} (i.e. not semi-stable) linearisation. We have two  reasons for this choice: 
\begin{itemize}
\item We have a natural way to frame such affine bundles by adding a framing parameter which varies, as in Plechinger's formalism, in a $\C^*$-torsor: we will consider pairs $(A,\varepsilon)$ consisting  of a rank 2-affine bundle on $C$ and an isomorphism $\varepsilon: \C\textmap{\simeq}  D_{L(A),x_0}$, where  $D_{L(A)}$ is the maximal destabilising line bundle of the linearisation $L(A)$, i.e. the first term of its Harder-Narasimhan filtration (see \cite[Theorem 2.5.2.3]{Sch}  and \cite[I.1.3]{HuLe} for the algebraic framework). 
\item In the last section we will consider affine rank 2-bundles on $\P^1_\C$, which admits no stable rank 2 vector bundles. 
\end{itemize}

Note that for such affine bundles we have a natural non-degeneracy condition: an affine rank 2 bundle with unstable linearisation on $C$ will be called {\it non-degenerate}  if the image of the invariant $h_A\in H^1(X, {\cal L}(A))$ of $A$  in the 1-cohomology of the minimal destabilising quotient ${\cal Q}_{L(A)}\edf {\cal L}(A)/{\cal D}_{L(A)}$ of the locally free sheaf ${\cal L}(A)$ does not vanish. We will study in detail the moduli space -- regarded as a topological space -- of framed non-degenerate holomorphic affine rank 2-bundles on $\P^1_\C$ whose linearisation, viewed as locally free sheaf,  is isomorphic to    $ {\cal O}_{\P^1_\C}(n_1)\oplus {\cal O}_{\P^1_\C}(n_2)$ where $n_1>n_2$.  We will see that this moduli space can be identified with the topological cokernel (see section \ref{coker-appendix}) of a   morphism of linear spaces   over the projective space $\P(\C[X_0,X_1]_{l})$ of binary forms of degree $l\edf -2-n_2$. In particular (although non-Hausdorff), this moduli space fibres naturally over this projective space with vector spaces as fibres. We will give an explicit formula for the dimension of the fibre over a point $[P]\in \P(\C[X_0,X_1]_{-2-n_2})$ in terms of $d\edf n_1-n_2$ and the {\it cactus rank} $\crk(P)$ of $P$, which is a fundamental invariant in the theory of binary forms  (see section \ref{apolar-section} in the appendix). 

Our result highlights an unexpected relationship between the classification problem for affine rank 2-bundles over $\P^1_\C$ and the theory of apolar ideals (see section \ref{apolar-section}),  and  also gives a strong motivation for studying and describing geometrically the cactus rank stratification of the projective space  $\P(\C[X_0,X_1]_{l})$ of binary forms of degree $l$ (see section \ref{cactus-rk-strat-small-l}).

\subsection{Preliminaries and main results}\label{prelim}
Let $\Aff(r,\C)$ be the group of affine transformations of $\C^r$.
Let $A$ and $X$ be  complex manifolds, and  $\pi: A \longrightarrow X$  a holomorphic map. A local trivialisation with standard fibre $\C^r$ of $\pi$ is a biholomorphism
 $$
 \tau=(\tau_1,\tau_2): \pi^{-1}(U_\tau) \longrightarrow U_\tau \times \mathbb{C}^r
 $$
 such that $\tau_1=\pi|_{\pi^{-1}(U)}$. This implies that the restriction $\tau_x\coloneq \tau_2|_{A_x}:A_x\to\C^r$ is a biholomorphism. Two such local trivialisations $\tau$, $\theta$   are said to be $\Aff(r,\C)$-compatible, if for every $x \in U_\tau \cap U_\theta$ the map 
 $$\tau_x \circ \theta_x^{-1}: \mathbb{C}^r \longrightarrow \mathbb{C}^r$$
 is an affine transformation, i.e. it belongs to $\Aff(r,\C)$. If this is the case, we obtain a comparison map
 $$
 U_\tau \cap U_\theta\to \Aff(r,\C),
 $$ 
 which is holomorphic.

Following the general formalism present in \cite[section 1.2.1]{Sch} we define: A bundle atlas with structure group $\Aff(r,\C)$ on $\pi$ is a set  $\mathcal{T}$ of local trivialisations with standard fibre $\C^r$ of $\pi$ satisfying the properties:

\begin{enumerate}
\item $\bigcup_{\tau\in {\cal T}} U_\tau=X$.
\item Any two elements of ${\cal T}$ are $\Aff(r,\C)$-compatible.
\end{enumerate}

\begin{dt} 
A holomorphic affine bundle of rank $r$ on $X$ is a couple 
$$(\pi: A \to X, \mathcal{T})$$
 consisting of a holomorphic map $\pi$ and a bundle  atlas ${\cal T}$ with structure group $\Aff(r,\C)$ on $\pi$ which  is maximal with respect to inclusion.
	
\end{dt}

The classical theory of  principal bundles in the  differential geometric framework (see for instance \cite{KN}, \cite{Te}) extends naturally in the holomorphic category. 
With any affine bundle $(\pi : A \to  X, \mathcal{T})$ one can associate an $\Aff(r,\C)$-principal  bundle as follows: for any $x\in X$  we define 
$$
P_{A,x}\coloneq \{f: \C^r \to  A_x|\ \hbox{ for any }\tau\in {\cal T}\hbox{ with }x\in U_\tau, \hbox{ we have }   \tau_x \circ f \in \Aff(r,\C)\}.
$$
The disjoint union $P_A\coloneq \coprod_{x\in X}P_{A,x}$
has a natural complex manifold structure and comes with an obvious $\Aff(r,\C)$-right action such that the natural map $P_\pi: P_A\to X$ becomes a holomorphic $\Aff(r,\C)$-principal bundle.

Conversely, for any holomorphic  principal $\Aff(r,\C)$-bundle $P$ on $X$ we define the  affine bundle  
$$
A^P\edf P\times_{\Aff(r,\C)}\C^r
$$
to be the bundle associated with $P$ and the standard  action of its structure group $\Aff(r,\C)$ on $\C^r$. Similarly to the case of vector bundles, we have 
%
%
\begin{re}
The functors $A\mapsto P_A$, $P\mapsto A^P$ define an   equivalence between the groupoid of affine rank $r$-bundles on $X$   and the groupoid of principal $\Aff(r,\C)$-bundles on $X$.	
\end{re}

We recall that the set of local sections of a principal bundle can always be identified with the set of its local trivialisations.

\begin{re} Let $(A,{\cal T})$ be a holomorphic  affine rank $r$ bundle on $X$, and  ${\cal S}$ be the set of local sections of the associated principal bundle $P_A$. The map ${\cal T}\to {\cal S}$ defined by
$$
{\cal T}\ni \tau\mapsto (U_\tau\ni x\mapsto \tau_x^{-1}\in P_{A,x})
$$
is a bijection.
\end{re}
Therefore ${\cal T}$ can be identified with the set of local trivialisations of $P_A$.
\vspace{1mm}\\

Let $(\pi:A\to X,{\cal T})$ be an affine bundle on $X$ and $x\in X$. A biholomorphic automorphism $u$ of $A_x$ will be called {\it translation} if there exists (or, equivalently, if for any) $\tau\in {\cal T}$ with $x\in U_\tau$ the composition $\tau_x \circ u \circ \tau_x^{-1}$ is a translation of $\C^r$. The set $L(A)_x$ of translations of $A_x$ has the natural structure of an $r$-dimensional complex vector space, and the natural  action of $L(A)_x$ on $A_x$ is free and transitive. In other words $A_x$  has the natural structure of a  complex affine space with associated (model) vector space $L(A)_x$.

The disjoint union $L(A)\edf \coprod_{x\in X} L(A)_x$
has the natural structure of a holomorphic vector bundle, which can be identified with the associated vector bundle $P_A\times_\rho \C^r$, where $\rho:\Aff(r,\C)\to\GL(r,\C)$ is the standard linear representation   which assigns to each affine transformation its linear part. 
 
\begin{dt}
The holomorphic vector bundle $L(A)$ will be called the linearisation of $A$, and its associated locally free sheaf will be denoted by ${\cal L}(A)$.
\end{dt}

We will also need the representation
$$
\tilde\rho:\Aff(r,\C)\to \GL(r+1,\C)
$$
of $\Aff(r,\C)$ defined by
$$
\tilde\rho(z\mapsto Mz+\zeta)\edf \begin{pmatrix}
1 &0\\ 
\zeta & M	
\end{pmatrix}.
$$
The image  $\tilde \Aff(r,\C)$  of this group monomorphism is a closed affine subgroup of $\GL(r+1,\C)$. The associated holomorphic vector bundle
$$
\tilde L(A)\edf P_A\times_{\tilde \rho}\C^{r+1}
$$
is a holomorphic rank $(r+1)$-vector bundle on $X$ which comes with a canonical $\tilde \Aff(r,\C)$-structure (reduction of structure group from $\GL(r+1,\C)$ to $\tilde \Aff(r,\C)$), see for instance \cite[Définition 2.3.1]{Te}. We will denote by $\tilde {\cal L}(A)$ the locally free sheaf associated with $\tilde L(A)$.

We recall that 

\begin{re} \label{rem-p0-phi}
\begin{enumerate}
\item The subgroup 	$\tilde \Aff(r,\C)\subset \GL(r+1,\C)$ leaves invariant the projection
$$
p_0:\C^{r+1}\to \C, \ p_0\begin{pmatrix}
	z_0\\\vdots\\ z_r
\end{pmatrix}\edf z_0.
$$
\item Let $E$ be a holomorphic complex vector bundle of rank $r+1$ on $X$. The data of an 	$\tilde \Aff(r,\C)$-structure   on $E$ is equivalent to the data of a bundle epimorphism $\varphi:E\to X\times\C$. 
\end{enumerate}
\end{re}

Concerning (2): for a point $x\in X$, the linear epimorphism $\varphi_x:E_x\to \C$ is $p_0\circ \tau_x$, where $\tau:E_{U_\tau}\to U_\tau\times \C^{r+1}$ is a local vector bundle trivialisation of $E$ which is compatible with the fixed $\tilde \Aff(r,\C)$-structure.  By (1) it follows that $p_0\circ \tau_x$ is independent of the choice of $\tau$.

Applying  this remark  to $\tilde L(A)$ we obtain a canonical bundle epimorphism
$$
\varphi_A:\tilde L(A)\to X\times\C 
$$
induced by the canonical 	$\tilde \Aff(r,\C)$-structure of $\tilde L(A)$. We will denote by the same symbol the associated sheaf epimorphism $\tilde {\cal L}(A)\to {\cal O}_X$.

\begin{dt}
An augmented rank $r+1$ holomorphic vector bundle on $X$ is a  	pair  $(E,\varphi)$ consisting of a holomorphic rank $r+1$-vector bundle $E$ on $X$  and a holomorphic bundle epimorphism $\varphi:E\to  X\times\C$ onto the trivial line bundle $X\times\C$ on $X$.
\end{dt}

Following \cite{Pl1}, \cite{Pl2} we will denote by $E_X^r$ the category of augmented rank $r+1$ holomorphic vector bundles on $X$.

\begin{re}\label{A-ag} Let $\ag=(E,\varphi)$ be an   augmented rank $r+1$ holomorphic vector bundle on $X$. Then the natural map $A_\ag\edf \varphi^{-1}(X\times\{1\})\to X$ comes with a natural structure ${\cal T}_\ag$ of a holomorphic affine bundle of rank $r$ on $X$ whose associated linearisation is canonically isomorphic to $\ker(\varphi)$.	
\end{re}

\begin{proof} Note first that the fibre  $A_{\ag,x}$ of $A_\ag$ over a point $x\in X$ is an  affine hyperplane in $E_x$, so it comes with a natural affine space space structure with model space $\ker(\varphi)_x$.

We define ${\cal T}_{\ag}$ to be the set of holomorphic local trivialisations $\tau$ of $A_\ag$ with the property that for any $x\in U$ the map $\tau_x:A_{\ag,x} \to \C^r$ is affine, i.e. it satisfies the identity
$$
\tau_x(\lambda u+(1-\lambda)v)=\lambda\tau_x(u)+(1-\lambda)\tau_x(v)
$$
for $u$, $v\in A_{\ag,x}$, $\lambda\in\C$. It is easy to check that ${\cal T}_\ag$ is a maximal bundle atlas  with structure group $\Aff(r,\C)$ on $A_\ag$.	
\end{proof}
Moreover one has (see \cite{Pl1}, \cite{Pl2}):
\begin{pr} \label{AffBdls-AugmBdls}
The functors $A\mapsto (\tilde L(A),\varphi_A)$, $\ag\mapsto 	(A_\ag,{\cal T}_\ag)$ define an equivalence of groupoides between the groupoid  of holomorphic affine rank $r$-bundles on $X$ and the groupoid of augmented holomorphic rank $(r+1)$-vector bundles on $X$.

Via the first functor the linearisation $L(A)$ of an affine bundle $A$ is canonically isomorphic with $\ker(\varphi_A)$. 
\end{pr}

We will denote by $(e_A)$ the  obtained extension of locally free sheaves
$$
0\to \ker(\varphi_A)={\cal L}(A)\hookrightarrow \widetilde{{\cal L}}(A)\textmap{\varphi_A}  {\cal O}_X\to 0 \eqno{(e_A)}
$$
which will play an importent role in this article. We will denote by 
$$h_A\in H^1(X,{\cal L}(A))=\Ext^1({\cal O}_X,{\cal L}(A))$$
 the  invariant of $(e_A) $ regarded as extension of ${\cal O}_X$ by ${\cal L}(A)$ (see section \ref{ClassTh-section} for details).

\vspace{2mm}

The main results of this paper are: 
\vspace{2mm}

  In section \ref{ClassTh-section} we will prove  the following classification theorem (see  \cite[Theorem 9]{Pl2} for the case $r=1$):

\begin{thr} \label{ClassTh-intro}
Let $X$ be a complex manifold, $E$ a holomorphic vector bundle of rank $r$ on $X$, ${\cal E}$ the associated locally free sheaf and $\Aut(E)=\Aut({\cal E})$ its automorphism group. The map $A\mapsto h_A \in H^1(X,{\cal L}(A))$ induces a bijection  between the set of isomorphism classes of holomorphic affine  rank $r$-bundles $A$  on $X$ with $L(A)\simeq E$ and the quotient $H^1(X,{\cal E})/\Aut(E)$.
\end{thr}

In this article, by the moduli space of affine bundles on $X$ with fixed linearisation type $[E]$ we will mean the topological quotient space  $H^1(X,{\cal E})/\Aut(E)$. In a future article we will come back to this statement in the framework of stack theory and compare  the moduli stack of affine bundles with fixed linearisation type $[E]$ and the quotient stack $[H^1(X,{\cal E})/\Aut(E)]$. 

This theorem shows that the classification of holomorphic affine bundles   is much more interesting and complex  than the   classification of holomorphic vector bundles. We will illustrate this contrast in the simplest case $X=\P^1_\C$; in this case the classification of vector bundles  is completely solved by Grothendieck's theorem  and the answer is very simple. The classification of affine bundles on $\P^1_\C$ is much more complex, see  Proposition \ref{general-case-prop}. 

Our next result concerns the classification of (framed) non-degenerate holomorphic  affine bundles $(A,\varepsilon)$ with $L(A)\simeq E$ over a Riemann surface $C$, where $E$ is a fixed unstable rank 2 holomorphic vector bundle on $C$, and $\varepsilon$ is a linear isomorphism  $\C\textmap{\simeq} D_{L(A),x_0}$ onto the fibre of the maximal destabilising line bundle of $E$ at a fixed point $x_0\in C$. Let $D_E$ be the maximal destabilising line bundle of $E$, $\Aut(E)_0$ be the closed subgroup of $\Aut(E)$ consisting of automorphisms which induce the identity on $D_{E}$, ${\cal E}$  the locally free sheaf associated with $E$ and $H^1(C,{\cal E})_0\subset H^1(C,{\cal E})$ the open subspace of cohomology classes which are mapped non-trivially to $H^1(C,{\cal E}/{\cal D}_E)$. Using the terminology introduced in section \ref{AffineRieamannSect}, we will prove:

\begin{thr}
\label{ClassTh2-intro}
Let $C$ be a Riemann surface,  $E$ be a fixed rank 2 unstable vector bundle on $C$, ${\cal E}$ the associated locally free sheaf, $x_0\in C$ and $\varepsilon_0:\C\to D_{E,x_0}$ be a fixed isomorphism. Then: 
\begin{enumerate}
	\item The map $A\mapsto h_A$ induces a bijection  between the set of isomorphism classes of non-degenerate affine holomorphic rank $r$-bundles $A$  on $X$ with $L(A)\simeq E$ and the quotient $H^1(C,{\cal E})_0/\Aut(E)$. 
	\item For an $x_0$-framed affine bundle $(A,\varepsilon)$ with $L(A)\simeq E$ choose an isomorphism $f: L(A)\to E$ such that $f\circ \varepsilon=\varepsilon_0$. 

The map $(A,\varepsilon)\mapsto H^1(f)(h_A)$ induces  a bijection  between the set of isomorphism classes of $x_0$-framed (non-degenerate) affine holomorphic rank $r$-bundles $A$  on $C$ with $L(A)\simeq E$ and the quotient $H^1(C,{\cal E})/\Aut(E)_0$ (respectively  $H^1(C,{\cal E})_0/\Aut(E)_0$). 
	\end{enumerate}
\end{thr}

In this article, by the moduli space of  framed  (non-degenerate)  affine holomorphic bundles with fixed linearisation type on $C$ we will mean the topological space obtained by endowing the corresponding moduli set with the topology induced by the identifications given by Theorem \ref{ClassTh2-intro}.

Note that in general the groups $\Aut(E)$, $\Aut(E)_0$ are affine algebraic groups (see Proposition \ref{Aut(E)AffAlg}), but not necessarily reductive. Therefore the classical GIT methods do not apply.

Section \ref{P1section} is dedicated to the classification of framed, non-degenerate affine rank 2 bundles  over $\P^1_\C$ whose linearisation, viewed as locally free sheaf,  is isomorphic to   $ {\cal O}_{\P^1_\C}(n_1)\oplus {\cal O}_{\P^1_\C}(n_2)$ where $-2\geq n_1>n_2$.
 
Making use of Theorem \ref{ClassTh2-intro}, we will show that the moduli space of isomorphism classes of such bundles, regarded as a topological space, can be identified with the topological cokernel (see section \ref{coker-appendix}) $\Coker(H)$ of a morphism 
$$H:L\to L'$$
 of linear spaces over the projective space $ \P(\mathbb{C}[X_0,X_1]_{-2-n_2}^*)$. $\Coker(H)$ can be viewed as the underlying topological space of the relative quotient stack (over $\P(\mathbb{C}[X_0,X_1]_{-2-n_2}^*)$) of the linear space $L'$ by $L$, where the latter acts on $L'$ via $H$. In a future article we will compare the moduli stack of the framed, non-degenerate affine bundles that we consider with the relative quotient stack of $L'$ by $L$. 
 Since the topological cokernel of a morphism of linear space plays an important role in this article, is of general interest, and (to my knowledge) is not treated in the literature, I have dedicated to it a section in the appendix.

  Although in general non-Hausdorff, our moduli space $\Coker(H)$ fibres over the projective space $\P(\mathbb{C}[X_0,X_1]_{-2-n_2}^*)$ with vector spaces as fibres. Our next goal will be understanding geometrically the stratification of   $\P(\mathbb{C}[X_0,X_1]_{-2-n_2}^*)$  defined by the level sets of the fibre dimension map $\xi\mapsto \dim(\Coker(H)_\xi)$. 
  
  Writing $\xi=[\Phi]$ for a non-trivial linear form $\Phi\in \mathbb{C}[X_0,X_1]_{-2-n_2}^*$, we will see that the fibre $\Coker(H)_\xi$ is naturally identified with the cokernel of the linear map
\begin{equation}\label{map-to-study}
 \mathbb{C}[X_0,X_1]_{n_1-n_2}\ni Q\mapsto \trp{m}_Q(\Phi)\in \mathbb{C}[X_0,X_1]_{-2-n_1}^*,	\end{equation}
 where, for a polynomial $Q\in \mathbb{C}[X_0,X_1]_{n_1-n_2}$, we have denoted by
 $$m_Q: \mathbb{C}[X_0,X_1]_{-2-n_1}\to \mathbb{C}[X_0,X_1]_{-2-n_2}$$
the multiplication by $Q$, and by $\trp{m}_Q$   its  transpose (algebraic adjoint).  

 The main idea for computing the dimensions of the cokernels of the family of linear maps (\ref{map-to-study}): for $l\in \N$ identify the dual space  $\mathbb{C}[X_0,X_1]_l^*$ with   $\mathbb{C}[X_0,X_1]_l$ via the isomorphism $D_l:\mathbb{C}[X_0,X_1]_l\textmap{\simeq}\mathbb{C}[X_0,X_1]_l^*$ which maps a polynomial  $P\in \mathbb{C}[X_0,X_1]_l$ to the differential operator $P(\delta_0,\delta_1)$, where $\delta_i\edf\frac{\partial}{\partial X_i}$. 
 
 Using these identifications we will prove (see section \ref{fibres-apolar-section}) the following duality theorem, which makes the connection between our classification problem for rank 2 affine bundles over $\P^1_\C$ and the theory of apolar ideals:

\begin{thr}{(Duality Theorem)} \label{duality-th-intro}
 Via the isomorphisms $D_l$, $D_{l-d}$ we have  
$$
\trp{m}_Q(P(\delta_0, \delta_1)) = Q\left(\frac{\partial}{\partial \delta_0}, \frac{\partial}{\partial \delta_1}\right)(P(\delta_0, \delta_1)).
$$
for any $P(\delta_0, \delta_1)\in \C[\delta_0,\delta_1]_l$, $Q\in \C[X_0,X_1]_d$.

In other words, putting $\Phi\edf D_l(P(\delta_0,\delta_1))\in \C[X_0,X_1]_l^*$, the differential operator associated with $\trp{m}_Q(\Phi)$ via $D_{l-d}^{-1}$ is obtained by applying the differential operator $Q\left(\frac{\partial}{\partial \delta_0}, \frac{\partial}{\partial \delta_1}\right)$ to   $P(\delta_0, \delta_1)$, {\it regarded as polynomial in the abstract variables $\delta_0$, $\delta_1$}.
\end{thr}

Put $l\edf -2-n_2$, $d\edf n_1-n-2$ and let $P\in \C[X_0,X_1]_l$ be the polynomial associated with $\Phi$ via $D_l^{-1}$.  The duality theorem  shows that the vector space $\coker(Q\mapsto \trp{m}_Q(\Phi))$, in which we are interested, can be identified with the cokernel
$$
C^d_P\edf \coker \big (\C[X_0,X_1]_d\ni Q\mapsto Q(\frac{\partial}{\partial X_0},\frac{\partial}{\partial X_1})(P)\big).
$$

At this stage we will make use essentially of the classical theory of binary forms and of the fundamental concepts  of this theory: apolar ideals  and cactus rank. The needed definitions and results are explained in the appendix (see section \ref{apolar-section}).
This theory belongs to  classical algebra, so we decided to present it in  the appendix.

The kernel $\ker\big(\C[X_0,X_1]_d\ni Q\mapsto Q(\frac{\partial}{\partial X_0},\frac{\partial}{\partial X_1})(P)\big)$ is just  the degree $d$ homogeneous component $\Ann(P)_d$ of the apolar ideal $\Ann(P)$ of $P$ (see Appendix \ref{apolar-section}). Its dimension depends on $d$ and the cactus rank $\crk(P)$ of $P$ (see Definition \ref{CactusRkDef}, Proposition \ref{d1=cactus}, Remark \ref{dim(Ann(P))Rem}).  We will prove:

\begin{pr}  The dimension $s:=\dim(C^d_P)$ depends only on $d$ and the cactus rank $r=\crk(P)$:

\begin{enumerate}
\item For $0\leq d\leq [\frac{l}{2}]$, the function $s=s(r)$ is given by:

$$s(r)=\left\{
\begin{array}{lcl}
l-d-(r-1)&\rm if & 1\leq r\leq d+1,\\
l-2d&\rm if &d+1\leq r\leq [\frac{l}{2}]+1.
\end{array}\right.$$ 

\item For $[\frac{l}{2}]\leq d\leq l$, the function $s=s(r)$ is given by: 

 $$s(r)=\left\{
\begin{array}{lcl}
l-d-(r-1)&\rm if & 1\leq r\leq l-d+1,\\
0&\rm if &l-d+1\leq r\leq [\frac{l}{2}]+1.
\end{array}\right.$$ 
	
%
\end{enumerate}
 
 \end{pr}

It follows that the stratification of $\P(\mathbb{C}[X_0,X_1]_{l}^*)\simeq \P(\mathbb{C}[X_0,X_1]_{l})$  defined by the level sets of the fibre dimension map $\xi\mapsto \dim(\Coker(H)_\xi)$ of our family of vector spaces is determined by $d$ and the cactus rank stratification of $\P(\mathbb{C}[X_0,X_1]_{l})$. Corollary \ref{stratif-comparison} gives explicit formulae for the level sets of the fibre dimension map in terms of $d$ and the level sets of the cactus rank. 

In section \ref{cactus-rk-strat-small-l} of the appendix we will describe explicitly the cactus rank stratification of the projective space $\P(\mathbb{C}[X_0,X_1]_{l})$ for $l\leq 4$. The case $l=4$ shows already that, in general,  $\dim(C^d_P)$ is not determined  only by $d$ 	and  the multiplicities of the roots of $P$, but depends effectively of the GIT invariants of $P$.

\section{The classification theorem for affine bundles}\label{ClassTh-section}

 Recall (see for instance \cite[Proposition 2 p. 184]{At}, \cite[Proposition 6.3 p. 234]{Ha}) that, in general, for any coherent sheaf ${\cal F}$ on $X$, isomorphism classes of extensions
 $$
 0\to {\cal F}\textmap{j} {\cal F}'\to {\cal O}_X\to 0 \eqno{(e)}
 $$
  of ${\cal O}_X$ by ${\cal F}$ correspond bijectively to elements of $\Ext^1({\cal O}_X,{\cal F})=H^1(X,{\cal F})$ via the extension class (invariant) map
  $$
  (e)\mapsto h_e\coloneq \delta(1),
  $$  
where $\delta:H^0(X,{\cal O}_X)\to H^1(X,{\cal F})$ is the connecting morphism associated with the exact sequence $(e)$.  In particular we define:
\begin{dt}\label{DefhA}
Let $A$ be an affine bundle over a complex manifold $X$ and $(e_A)$ the associated short exact sequence (see \ref{prelim}). We define the cohomology invariant of $A$ by  $h_A\edf h_{e_A}\in H^1(X,{\cal L}(A))$.	
\end{dt}

Using the functoriality of the connecting morphism with respect to morphisms of short exact sequences, it follows easily:

\begin{re}\label{Functor-h}
Consider an exact sequence of the form $(e)$, a sheaf isomorphism  $f:{\cal G}\to {\cal F}$ and the exact sequence
$$
 0\to {\cal G}\textmap{j\circ f} {\cal F}'\to {\cal O}_X\to 0 \eqno{(f^*(e))}
 $$
 
 Then
 $$
 h_{f^*(e)}=H^1(f)^{-1}(h_e).
 $$
\end{re}

We will apply the general results explained above to the extensions  associated  with affine bundles  (see section  \ref{prelim}).


\begin{lm}\label{nice-lemma}
Let $A$, $B$ be affine bundles on $X$ and $g:L(A)\to L(B)$ be a vector bundle isomorphism such that $H^1(g)(h_{A})=	h_{B}$. There exists an affine bundle isomorphism $\ \gamma:A\to B$ such that $L(\gamma)=g$.
\end{lm}
\begin{proof}

 By Remark \ref{Functor-h} we know that the extension invariant of
 $$
 0\to {\cal L}(A)\textmap{\iota_B\circ g} \tilde {\cal L}(B)\textmap{\varphi_B} {\cal O}_X\to 0 \eqno{(g^*(e_B))}
 $$	
 is  $H^1(g^{-1})(h_{B})=h_{A}$. Therefore this extension is isomorphic to $(e_A)$, so there exists a commutative diagram
 \begin{center}

\begin{tikzcd}[column sep=1.2cm]
0 \arrow[r, ""] & {\cal L}(A) \arrow[r, "\iota_A ", hook] \arrow[d, "\id"] & \tilde {\cal L}(A) \arrow[d, "\tilde\gamma ", "\simeq"{anchor=south, rotate=90}] \arrow[r, "\varphi_A"] & \Oc_X  \arrow[d, "\id"] \arrow[r, ""]  & 0 \phantom{.} \\
0 \arrow[r, ""] & {\cal L}(A)  \arrow[r, "\iota_B\circ g "] & \tilde {\cal L}(B)  \arrow[r, "\varphi_B"] & \Oc_X \arrow[r, ""] &  0.
\end{tikzcd}

\end{center}
  Using Remark \ref{A-ag} note now that $\tilde\gamma$ induces an isomorphism of affine bundles
$$
\gamma:A=\varphi_A^{-1}(1)\textmap{\simeq}\varphi_B^{-1}(1)=B.
$$ 
 The same diagram shows that the isomorphism $\ker(\varphi_A)={\cal L}(A)\to \ker(\varphi_B)={\cal L}(B)$  induced by $\tilde\gamma$ coincides with $g$.
 
\end{proof}

We can prove now Theorem \ref{ClassTh-intro}  stated in the introduction:


\begin{proof}   (of Theorem \ref{ClassTh-intro}):


For the surjectivity, let  $\beta \in H^1(X,{\cal E})$. By Atiyah's classification theorem mentioned above, there exists an exact sequence 
$$
0\to {\cal E}\textmap{j} {\cal E}'\textmap{\varphi} {\cal O}_X\to 0 \eqno{(e)}
$$
with $h_e=\beta$. Let $E'$ be the vector bundle associated with ${\cal E}'$ and use the same symbol for the  bundle  epimorphism $E'\to X\times\C$ associated with $\varphi$. Put $\ag\edf (E',\varphi)$. By Remark \ref{A-ag} it follows that $A_\ag=\varphi^{-1}(X\times\{1\})$ is an affine bundle whose linearization $L(A_\ag)$ is identified with $\ker(\varphi)$. But $j$ induces an isomorphism     $E\textmap{\simeq}\ker(\varphi)$.

For the injectivity, let $A$, $B$ be two affine bundles of rank $r$ on $X$ with $L(A)\simeq E$, $L(B)\simeq E$ and fix vector bundle isomorphisms $f:E\textmap{\simeq} L(A)$, $g:E\textmap{\simeq} L(B)$. We have to prove that, assuming that $H^1(f^{-1})(h_{A})$,  $H^1(g^{-1})(h_{B})$ belong to the same  $\Aut(E)$-orbit in $H^1(X,{\cal E})$, then $A\simeq B$.

Choose a vector bundle automorphism $\psi\in \Aut(E)$ such that 
$$H^1(\psi^{-1})(H^1(f^{-1})(h_{A}))=H^1(g^{-1})(h_{B}).
$$
By Lemma \ref{nice-lemma}, it follows that there exists an affine bundle isomorphism $\gamma: A \to B$ such that $L(\gamma)=g\circ \psi^{-1}\circ f^{-1}$. Then $A$ is isomorphic to $B$.
\end{proof}

By Cartan's finiteness theorem we have:

\begin{re} If $X$ is a  compact complex manifold, then the space $H^1(X,{\cal E})$ intervening in Theorem \ref{ClassTh-intro} is a complex space of finite dimension. 
\end{re}

Moreover one has

\begin{pr}\label{Aut(E)AffAlg}
Let $X$ be a compact connected complex manifold, $E$ a holomorphic vector bundle on $X$ and ${\cal E}$ the associated locally free sheaf. Then $\Aut(E)=\Aut({\cal E})$ is an affine algebraic group.
\end{pr}

\begin{proof} We will prove that $\Aut({\cal E})$ is an affine Zariski open (in the algebraic sense) subset of the finitely dimensional complex vector space $H^0(X,{\cal E}nd(E))$, and that the composition  and the inversion maps are regular.

 Let  ${\cal E}nd(E)$ be the sheaf of  (local) holomorphic sections of the bundle $\End(E)$.  For any $f \in H^0(X,{\cal E}nd(E))$, the determinant map $\det(f)\in {\cal O}(X)$ is constant, because  $X$ is compact  and connected. Fix $x_0\in X$. A section $f\in H^0(X,{\cal E}nd(E))$ is an automorphism of $E$ if and only if $\det(f)(x_0)\ne 0$. Therefore
$$
\Aut(E)=H^0(X,{\cal E}nd(E))\setminus \delta_{x_0}^{-1}(0),
$$
where $\delta_{x_0}: H^0(X,{\cal E}nd(E))\to \C$ is  defined by
$$
\delta_{x_0}(f)=\det(f)(x_0).
$$
In the commutative diagram 
 
$$\begin{tikzcd}
H^0(X,{\cal E}nd(E)) \arrow[rr, "\ev_{x_0}"] \ar[rd, dotted, "\delta_{x_0}", swap] && \End(E_{x_0}) \ar[ld, "\det"] \\ & \C  
\end{tikzcd}$$
the evaluation map $\ev_{x_0}$ at $x_0$ is linear, which proves that $\delta_{x_0}(f)$
 is a  (homogeneous) polynomial map of degree $r$, so (by \cite[Exercice 2.1 p. 79]{Ha}) the group $\Aut({\cal E})$ is an open affine subvariety of the vector space $H^0(X,{\cal E}nd(E))$, as claimed.
 
 The composition map $\Aut(E)\times \Aut(E)\to \Aut(E)$ is the restriction to $\Aut(E)$ of the bilinear map $H^0(X,{\cal E}nd(E))\times H^0(X,{\cal E}nd(E))\to H^0(X,{\cal E}nd(E)) $ given by composition of endmorphisms, so it is regular. 
 
 For the inversion map: Via the canonical isomorphism ${\cal E}\textmap{\simeq} (\wedge^{r-1}{\cal E})^\vee \otimes \det({\cal E})$ (see \cite[Exercice 5.16, section II.5]{Ha}) the inversion map $f\mapsto f^{-1}$ is given by $f\mapsto \trp{(\wedge^{r-1} (f))}\det(f)^{-1}$. It suffices to note that $f\mapsto \trp{(\wedge^{r-1} (f))}$ is the restriction to $\Aut({\cal E})$ of a homogenous polynomial map of degree $r-1$ on the vector space $H^0(X,{\cal E}nd(E))$ and, again by \cite[Exercice 2.1 p. 79]{Ha},  $f\mapsto \det(f)^{-1}$ is a regular function on the affine subvariety $\Aut({\cal E})$.

 \end{proof}

 \subsection{Affine bundles of rank 2 over Riemann surfaces} \label{AffineRieamannSect}
 
 Let $C$ be a Riemann surface. We have two classes of affine bundles of rank $r$ on $C$: 
 
 \begin{enumerate}
 	\item The affine bundles $A$ with $L(A)$ semi-stable,
 	\item The affine bundles $A$ with $L(A)$ unstable.
 \end{enumerate}

Since we are especially interested in the case $C=\P^1_\C$, we will focus on the second class.
For an unstable bundle $E$ of rank 2 we will denote by $D_E\subset E$ the maximal destabilising line subbundle of $E$ and by
$$
p_E:E\to Q_E\edf E/D_E
$$
the canonical epimorphism onto the corresponding quotient. 

\begin{re}\label{Invariance-of-DE}
Let $f : F \to E$ be an isomorphism of unstable vector  bundles. Then we have $f(D_F)=D_E$, so $f$ induces an isomorphism  $\tilde f : Q_F \to Q_E$ with the property
$$\tilde f \circ p_F=p_E\circ f. $$

\end{re}

We define:
\begin{dt} Let $A$ be an affine bundle of rank 2  on $C$ with $L(A)$ unstable, and let $x_0\in C$.
\begin{enumerate}
	\item A will be called non-degenerate 	if $H^1(p_E)(h_{A})$ is non-zero in $H^1(C,{\cal Q}_E)$.
	\item An $x_0$-framing of $A$ is an isomorphism $\varepsilon:\C\textmap{\simeq} D_{L(A),x_0}$.
\end{enumerate}

An $x_0$-framed affine bundle of rank 2 is an affine bundle  as above endowed with   an $x_0$-framing. An isomorphism $(A,\varepsilon)\to (A',\varepsilon')$ of $x_0$-framed affine  bundles is an isomorphism $\psi:A\to A'$ such that $L(\psi)\circ \varepsilon=\varepsilon'$.
\end{dt}

Let now $E$ be a fixed rank 2 unstable vector bundle on $C$. Our goal is the classification of isomorphism classes of affine bundle $A$ on $C$ with $L(A)\simeq E$. We will handle this problem in several steps:

\begin{enumerate}[(S1)]
	\item Classify all $x_0$-framed non-degenerate affine bundles $A$ with $L(A)\simeq E$.
	\item Classify all non-degenerate affine bundles $A$ with $L(A)\simeq E$.
	\item  Classify all  affine bundles $A$ with $L(A)\simeq E$
\end{enumerate}

Note that,  by Remark \ref{Invariance-of-DE}, any automorphism of $E$ leaves invariants its maximal destabilising line bundle. Let $\Aut(E)_0$ be the closed subgroup of $\Aut(E)$ consisting of automorphism which induce the identity on $D_{E}$.   Let also $H^1(C,{\cal E})_0\subset H^1(C,{\cal E})$ be the Zariski open subspace defined as follows
$$
H^1(C,{\cal E})_0\edf \{h\in H^1(C,{\cal E})|\ H^1(p_E)(h)\ne 0\}.
$$

Taking into account Theorem \ref{ClassTh-intro}, we can prove now Theorem \ref{ClassTh2-intro} stated in the introduction: 

%
%

\begin{proof} (of Theorem \ref{ClassTh2-intro}):

(1)	The first statement is obvious. 
  
(2) Note first that, under the assumption $L(A)\simeq E$, such an isomorphism $f$ always exists. 
	
	 The map 
$$
[(A,\varepsilon)]\mapsto [H^1(f)(h_{A})]_{\Aut(E)_0}
$$
between the set of isomorphism classes  of $x_0$-framed non-degenerate affine holomorphic rank $r$-bundles $A$  on $C$ with $L(A)\simeq E$ and $H^1(C,{\cal E})_0/\Aut(E)_0$ is obviously well defined. It is also clearly surjective by Theorem \ref{ClassTh-intro}.

For the injectivity, let now  $(A,\varepsilon)$, $(B,\sigma)$ be two $x_0$-framed non-degenerate affine bundles, and $f: L(A)\textmap{\simeq} E$, $g: L(B)  \textmap{\simeq}E$  vector bundle isomorphisms   satisfying  
$$f\circ \varepsilon=g \circ \sigma=\varepsilon_0,$$
such that   $H^1(f)(h_{B})$, $H^1(g)(h_{B})$ belong to the same $\Aut(E)_0$-orbit. 

Therefore, there exists an isomorphism $\psi \in \Aut(E)_0$ such that 
	 $$H^1(\psi)(H^1(f)(h_{A}))=H^1(g)(h_{B}). $$
Putting $b:=g^{-1}\circ\psi\circ f$ note that $H^1(b)(h_{A})=h_{B}$, so, by Lemma \ref{nice-lemma}, there exists an affine bundle isomorphism $\beta:A\to B$ such that $L(\beta)=b$. It suffices to note that $L(\beta)\circ\varepsilon=b\circ \varepsilon=\sigma$, so $\beta$ is an isomorphism of framed affine bundles $(A,\varepsilon)\to (B,\sigma)$.

\end{proof}

\section{Affine bundles over \texorpdfstring{$\P^1_\C$}{0}}\label{P1section}

\subsection{The general case}\label{general-case}
Let ${\cal E}$ be a locally free sheaf of rank $r$ on $\P^1_\C$. By Grothendieck's classification theorem (See \cite{Gro})  we may assume that 
$${\cal E}=\bigoplus_{i=1}^m \mathcal{O}_{{\P^1_\C}}(n_i)^{\oplus s_i},$$ 
where $n_i\in\Z$ such that
$$
n_1> \dots > n_m, \ r=\sum_{i=1}^m s_i,  
$$ 
which gives an identification
$${\cal E}^* \otimes {\cal E} =\bigoplus_{i,j=1}^m \mathcal{O}_{\P^1_\C}(n_i-n_j)^{\oplus s_i s_j}$$
An element of the space $\End({\cal E})=H^0(\P^1_\C,{\cal E}^* \otimes {\cal E})$ of global endomorphisms of ${\cal E}$ can be identified with  an upper triangular  matrix 
$$A=(A_{ij})_{1\leq i,j\leq m},$$
where 
$$A_{ij}\in M_{s_is_j}(H^0(\P^1_\C,{\cal O}_{\P^1_\C} (n_i-n_j)))=M_{s_is_j}(\C[X_0,X_1]_{n_i-n_j}).$$
Here we use the notation $\C[X_0,X_1]_d$ for the complex vector space of homogeneous polynomials of degree $d$ in the variables $X_0$, $X_1$. We will use the convention $\C[X_0,X_1]_d=0$ for $d<0$.

Note that the $i$-th diagonal element $A_{ii}$ of $A$ belongs to $M_{s_is_i}(\C)$ and an element $A\in \End({\cal E})$ is invertible if and only if $\det(A_{ii})\ne 0$ for $1\leq i\leq m$. Therefore
\begin{equation}\label{Aut(E)}
\begin{split}
 \Aut(\mathcal{E})= \big\{(A_{ij})_{1\leq i,j\leq m}\,  \vline \ A_{ij} \in {M}_{s_i s_j} (\mathbb{C}&[X_0,X_1]_{n_i-n_j}),\\
 &   A_{ii}\in\GL(s_i,\C) \hbox{ for }1\leq i\leq m\big\}. 
 \end{split}
\end{equation}
This is obviously an affine  algebraic group, which confirms the general  Proposition \ref{Aut(E)AffAlg} proved above.

 Taking into account the isomorphism $\mathcal{K}_{\mathbb{P}_\C^1}\simeq \mathcal{O}_{\mathbb{P}_\C^1}(-2)$, and using the Serre Duality theorem we obtain the following:

\begin{equation}\label{H1E}
\begin{split}
   H^1(X,\mathcal{E})=& \bigoplus_{i=1}^m H^1(\mathbb{P}^1_\C,\mathcal{O}_{\P^1_\C}(n_i))^{\oplus s_i}    
 \simeq \bigoplus_{i=1}^m H^0(\mathbb{P}^1_\C,  {\mathcal{O}_{\P^1_\C}(-n_i-2))^*}^{\oplus s_i} \\
\simeq  &\bigoplus_{i=1}^m (\mathbb{C}[X_0,X_1]_{l{_i}}^*)^{\oplus s_i},  
  \end{split}
\end{equation}
where $l_i\edf -2-n_i$.

Taking into account   Theorem \ref{ClassTh-intro}, we obtain the following general result which solves the classification problem of affine bundles of arbitrary rank on $\P^1_\C$: 

\begin{pr} \label{general-case-prop} The moduli space of  affine bundles on $\P^1_\C$ with fixed linearisation type $[\bigoplus_{i=1}^m \mathcal{O}_{\mathbb{P}^1}(n_i)^{\oplus s_i}]$ is naturally identified with the quotient
$$
 \bigoplus_{i=1}^m (\mathbb{C}[X_0,X_1]_{-2-n_i}^*)^{\oplus s_i}\big / G$$
where  $G=\Aut(\bigoplus_{i=1}^m \mathcal{O}_{\mathbb{P}^1}(n_i)^{\oplus s_i})$ is the affine algebraic group 
$$
G=\big\{(A_{ij})_{1\leq i,j\leq m}\,  \vline \ A_{ij} \in {M}_{s_i s_j} (\mathbb{C} [X_0,X_1]_{n_i-n_j}), 
     A_{ii}\in\GL(s_i,\C) \hbox{ for }1\leq i\leq m\big\} 
$$
acting on $ \bigoplus_{i=1}^m (\mathbb{C}[X_0,X_1]_{-2-n_i}^*)^{\oplus s_i}$ as follows:  for any $A= (A_{ij})_{1\leq i,j\leq m}\in G$ and $\Phi\edf (\Phi_i)_{1\leq i \leq m}\in H^1(X,\mathcal{E})$,     
$$
A \bullet \Phi = (\sum_{j=1}^m \trp{m}_{A_{ij}}(\Phi_j))_{1\leq i \leq m}.
$$
Here $m_{A_{ij}}$ stands for the multiplication map 
$\C[X_0,X_1]_{-2-n_i}\to \C[X_0,X_1]_{-2-n_j}$
by the matrix $A_{ij}$, and $\trp{m}_{A_{ij}}$ stands for its transpose.
\end{pr}

This result shows already how difficult is the classification  of affine bundles compared to the classification of vector bundles. 
Note that this quotient is in general highly non-Hausdorff and difficult to describe geometrically in full generality.

\subsection{The case \texorpdfstring{$r=2$}{1}. General classification results}

In this section we will specify the action of the group $\Aut(\mathcal{E})$ on $H^1(\mathbb{P}_\C^1,\mathcal{E})$ through the  identifications given in section \ref{P1section} and we will give explicit descriptions for the quotients intervening in Theorem \ref{ClassTh2-intro} (in the case $X=\P^1_\C$). 
 
\vspace{2mm}

To simplify the formalism, we will study in detail the case $r=2$ and $s_1=s_2=1$. We can assume ${\cal E}= {\cal O}_{\P^1_\C}(n_1)\oplus {\cal O}_{\P^1_\C}(n_2)$ with $n_1>n_2$. In other words we suppose that ${\cal E}$ is unstable (non-semistable). The maximal destabilising subsheaf of ${\cal E}$ is ${\cal O}_{\P^1_\C}(n_1)$.

We also assume that $-2>n_1$, so that the two cohomology spaces $H^1(\P^1_\C, {\cal O}(n_i))$ are both non-trivial. Putting $l_i:= -n_i-2>0$, we obtain  canonical identifications

$$
H^1({\P^1_\C},{\cal E})=\mathbb{C}[X_0,X_1]_{l_1}^*\oplus \mathbb{C}[X_0,X_1]_{l_2}^*, 
$$
\begin{equation}\label{formula-for-Aut}
\Aut({\cal E})\simeq G\edf \left\{\begin{pmatrix}
u_1 & Q\\ 0 &u_2	
\end{pmatrix}\,\vline\ u_1,\ u_2\in\mathbb{C}^*,\ Q\in \C[X_1,X_1]_d
 \right\}\hbox{ with } d:=l_2-l_1,
\end{equation}
and, via these identifications, $\Aut(\Ec)$ acts on $H^1({\P^1_\C},{\cal E})$ by
\begin{equation}\label{explicit-action}
\begin{pmatrix}
u_1 & Q\\ 0 &u_2	
\end{pmatrix}\begin{pmatrix}
\Phi_1\\ \Phi_2	
\end{pmatrix}=\begin{pmatrix}
u_1\Phi_1+\trp{m}_Q(\Phi_2)\\ u_2\Phi_2	
\end{pmatrix}	
\end{equation}
where $m_Q:\C[X_0,X_1]_{l_1}\to \C[X_0,X_1]_{l_2} $ denotes the linear map defined by multiplication with $Q$, and $\trp{m}_Q$ stands for its transpose.

Note that that the group $\Aut({\cal E})$ fits in the short exact sequence 

\begin{equation}\label{ShExSeqAut}
\{1\}\to  N:=\left\{\begin{pmatrix}
1 & Q\\ 0 &1	
\end{pmatrix}\,\vline \ Q\in \C[X_1,X_1]_d
 \right\}\hookrightarrow \Aut({\cal E})\to \C^*\times\C^*\to \{1\}  	
\end{equation}
whose kernel $N$ is isomorphic to $\C[X_1,X_1]_d$ via the obvious map.  Note that via the identification \ref{formula-for-Aut},   $\Aut({\cal E})_0$  corresponds to the subgroup $G_0\subset G$ defined by the condition $u_1=1$. By the classification  Theorems \ref{ClassTh-intro}, \ref{ClassTh2-intro}  we obtain
\begin{pr}\label{ClassifPropOnP1} Let $E$ be the underlying vector bundle of ${\cal E}={\cal O}_{\P^1_\C}(n_1)\oplus {\cal O}_{\P^1_\C}(n_2)$ with $-2>n_1>n_2$. Put $l_i\edf -2-n_i$, $d\edf l_2-l_1$. Fix $x_0\in\P^1_\C$. Let the groups 
\begin{equation}
\begin{split}
G\edf &\left\{\begin{pmatrix}
u_1 & Q\\ 0 &u_2	
\end{pmatrix}\,\vline\ u_1,\ u_2\in\mathbb{C}^*,\ Q\in \C[X_1,X_1]_d\right\},\\
G_0\edf &\left\{\begin{pmatrix}
1 & Q\\ 0 &u_2	
\end{pmatrix} \,\vline\   u_2\in\mathbb{C}^*,\ Q\in \C[X_1,X_1]_d\ \right\} 
\end{split}
\end{equation}
act  on $\mathbb{C}[X_0,X_1]_{l_1}^*\oplus \mathbb{C}[X_0,X_1]_{l_2}^*$ via formula (\ref{explicit-action}).
\begin{enumerate}
\item 	The moduli space of isomorphism classes of (non-degenerate) affine bundles $A$ on $\P^1_\C$ with $L(A)\simeq  E$ is homeomorphically identified with the quotient
\begin{equation}\label{quotients-def}
\begin{split}
{\cal Q}\edf &\qmod{\mathbb{C}[X_0,X_1]_{l_1}^*\oplus \mathbb{C}[X_0,X_1]_{l_2}^*}{G}, \hbox{ respectively }\\
{\cal Q}_0\edf &\qmod{\mathbb{C}[X_0,X_1]_{l_1}^*\oplus \big(\mathbb{C}[X_0,X_1]_{l_2}^*\setminus\{0\}\big)}{G}.
\end{split}
\end{equation}

\item The moduli space of isomorphism classes of $x_0$-framed (non-degenerate) affine $(A,\varepsilon)$ on $\P^1_\C$ with $L(A)\simeq  E$ is homeomorphically identified with the quotient
\begin{equation}\label{quotients-framed-def}
\begin{split}
\Qg\edf &\qmod{\mathbb{C}[X_0,X_1]_{l_1}^*\oplus \mathbb{C}[X_0,X_1]_{l_2}^*}{G_0}, \hbox{ respectively }\\
\Qg_0\edf &\qmod{\mathbb{C}[X_0,X_1]_{l_1}^*\oplus \big(\mathbb{C}[X_0,X_1]_{l_2}^*\setminus\{0\}\big)}{G_0}.
\end{split}
\end{equation}

\end{enumerate}
	
\end{pr}

The quotients ${\cal Q}$, ${\cal Q}_0$, $\Qg$, $\Qg_0$ will be endowed with the quotient topologies.

\subsection{The case \texorpdfstring{$r=2$}{2}.  Framed  non-degenerate  affine bundles}\label{Cokertopoiso}

By  Proposition \ref{ClassifPropOnP1} the set of isomorphism classes of  $x_0$-framed  non-degenerate  affine bundles $A$ on $\P^1_\C$ with $L(A)\simeq  E$ is identified with the quotient $\Qg_0$. The goal of this section is to describe explicitly this quotient endowed with its natural quotient topology.

We have the following diagram
\begin{equation}\label{1st-diag}
\begin{tikzcd}[column sep =-1mm]
(\mathbb{C}[X_0,X_1]_{l_2}^*	\setminus\{0\})\times \C[X_0,X_1]_d\ar[rd] \ar[rr, "\eta"]& &(\mathbb{C}[X_0,X_1]_{l_2}^*	\setminus\{0\})\times  \mathbb{C}[X_0,X_1]_{l_1}^*\ar[dl]\\
&\mathbb{C}[X_0,X_1]_{l_2}^*	\setminus\{0\}&
\end{tikzcd}
\end{equation}
of trivial linear spaces over $\C[X_0,X_1]_{l_2}^*\setminus\{0\}$, where $\eta$ is the fibrewise linear map
$$
\eta(\Phi_2,Q)=(\Phi_2,\trp{m}_Q(\Phi_2))
$$

Factorising all the spaces in the above diagram by the subgroup
$$
G_0^2\edf \left\{\begin{pmatrix}
1 & 0\\ 0 &u_2	
\end{pmatrix} \,\vline\   u_2\in\mathbb{C}^*  \right\}\simeq\C^*
$$
of $G_0$, we obtain the following diagram of linear spaces over the projective space
${\cal P}= \mathbb{C}[X_0,X_1]_{l_2}^*	\setminus\{0\}/\C^*=\P(\mathbb{C}[X_0,X_1]_{l_2}^*)$ 
\begin{equation}\label{2nd-diag}
\begin{tikzcd}[column sep =7mm]
\Theta_{\cal P}\otimes_\C \C[X_0,X_1]_d\ar[rd] \ar[rr, "H"]& &{\cal P}\times    \mathbb{C}[X_0,X_1]_{l_1}^*\ar[dl]\,,\\
&{\cal P}&
\end{tikzcd}
\end{equation}
where $\Theta_{\cal P}=|{\cal O}_{\cal P}(-1)|$ is the tautological line bundle of ${\cal P}$. Note that  we have an obvious set-theoretical identification  
$$
\Qg_0=\coprod_{[\Phi]\in {\cal P}}  \mathrm{coker} (Q\mapsto \trp{m}_Q(\Phi)).
$$

For the topology of $\Qg_0$ we have the following result:
\begin{pr}\label{homeo-Coker}
We have a natural homeomorphism $\Qg_0\textmap{\simeq}\Coker(H)$, where $\Coker(H)$ stands for the topological cokernel of $H$ (see Definition \ref{DefCoker} in the Appendix).
\end{pr}
\begin{proof}

Note first that the right hand space  ${\cal P}\times    \mathbb{C}[X_0,X_1]_{l_1}^*$ in (\ref{2nd-diag}) is precisely the quotient (in the topological sense) of $\mathbb{C}[X_0,X_1]_{l_1}^*\oplus \big(\mathbb{C}[X_0,X_1]_{l_2}^*\setminus\{0\}\big)$ by the equivalence relation, say $R$, induced by the $G_0^2$-action on this space.

Let $S$ be the equivalence relation on ${\cal P}\times    \mathbb{C}[X_0,X_1]_{l_1}^*$ defined  by
$$
S\edf \big\{\big(([\Phi],\phi)\,,\, ([\Phi],\phi+\trp{m}_Q(\Phi)))|\ [\Phi]\in{\cal P},\  Q\in  \C[X_1,X_1]_d\big\}.
$$

Note that the quotient space ${\cal P}\times    \mathbb{C}[X_0,X_1]_{l_1}^*/S$ is precisely the topological cokernel of $H$, by the definition of this cokernel. 

Using the elementary Proposition \ref{quotient-2steps} stated below, it suffices to prove that the equivalence relation $R_S$ on  $\mathbb{C}[X_0,X_1]_{l_1}^*\oplus \big(\mathbb{C}[X_0,X_1]_{l_2}^*\setminus\{0\}\big)$ is precisely the equivalence relation associated with the $G_0$-action on this space. This follows easily noting that the equivalence class of a pair $(\Phi,\phi)\in \mathbb{C}[X_0,X_1]_{l_1}^*\oplus \big(\mathbb{C}[X_0,X_1]_{l_2}^*\setminus\{0\}\big)$ with respect to $R_S$ coincides with its $G_0$-orbit.

\end{proof}
Note that Proposition \ref{homeo-Coker} cannot be obtained using Corollary \ref{quotient-2steps-coro} stated below, because the subgroup $G_0^2$ is not normal in $G_0$. Note however that making use of this corollary we obtain an alternative description of $\Qg_0$ as the $\C^*$-quotient of a topological cokernel on $\mathbb{C}[X_0,X_1]_{l_2}^*	\setminus\{0\}$.

\begin{pr}\label{quotient-2steps}
Let $X$ be a topological space, $R\subset X\times X$ an equivalence relation on $X$, $p_R:X\to Q\edf X/R$ the canonical surjection, $S\subset Q\times Q$ an equivalence relation on $Q$, and $p_S:Q\to  Q/S$  the canonical surjection.	 
\begin{enumerate}
\item The subset 
$$
R_S\edf \{(x,y)\in X\times X| (p_R(x),p_R(y))\in S\}\subset X\times X
$$	
is an equivalence relation on $X$. 
\item The composition $ p_S\circ p_R$ is compatible with $R_S$ and the induced map 
$$h:X/R_S\to    Q/S$$
is a homeomorphism with respect to the quotient topologies. 
\end{enumerate}

\end{pr}

\begin{co}\label{quotient-2steps-coro}
Let $X$ be a topological space, $G$ an arbitrary group and    
$$\alpha:G\times X\to X$$ 
be an action by homeomorphisms.  Let $N\subset G$ be a normal subgroup and $H\edf G/N$. The quotient $X/N$ comes with a $H$-action induced by $\alpha$ and the natural bijection
$$
h:X/G\to (X/N)/H
$$ 	
is a homeomorphism. 
\end{co}

\subsection{The duality theorem, the fibres of $\Coker(H)$ over ${\cal P}$ and apolar ideals} 
\label{fibres-apolar-section}

The  goal of this section is to  describe explicitly:
\begin{enumerate}
	\item The projective space ${\cal P}$, the base of the topological cokernel $\Coker(H)$.
	\item The fibres of $\Coker(H)$ over ${\cal P}$, and notably their dimensions.
\end{enumerate}

We will see that the computation of the dimensions of these fibres is related to the classical theory of apolar ideals.\\

The main idea of this section: we identify  the dual space $\C[X_0,X_1]_{l}^*$ with the space of constant-coefficient homogeneous differential operators of order $l$. More precisely we put 
$$\C[\delta_0,\delta_1]_l:=\big\{P(\delta_0,\delta_1)|\ P\in  \C[X_0,X_1]_{l}\big\},$$
where $\delta_i:=\frac{\partial}{\partial X_i}$ and $\delta_0^k \delta_1^{l-k}=\frac{\partial^d}{\partial^k X_0 \partial^{l-k} X_1}$, and we define the linear isomorphism $D_l: \C[\delta_0,\delta_1]_l\to  \C[X_0,X_1]_{l}^*$ by
assigning to $\Delta\in \C[\delta_0,\delta_1]_l$  the linear form
$$D_l(\Delta):  \C[X_0,X_1]_{l}\to \C$$
which maps a polynomial $P\in \C[X_0,X_1]_{l}$ to $\Delta(P)$; here  $\Delta$ is regarded as a differential operator.

\begin{re}
Note that
$$
D_l(\delta_0^k \delta_1^{l-k})(X_0^s X_1^{l-s})=\delta_0^k \delta_1^{l-k}(X_0^s X_1^{l-s})=\left\{
\begin{array}{ccc}
s! (l-s)! &\rm if & k=s,
\\
0 &\rm if & k\ne s.
\end{array}
\right.
$$
This formula compares, via $D_l$, the canonical basis of $\C[\delta_0,\delta_1]_l$ to the dual of the canonical basis  of $\C[X_0,X_1]_{l}$. The inverse of $D_l$ is given by the formula
$$
D_l^{-1}(\Phi)=\sum_{s=0}^l \frac{\Phi(X_0^s X_1^{l-s})}{s!(l-s)!} \delta_0^s\delta_1^{l-s}=\frac{1}{l!}\sum_{s=0}^l \binom{l}{s}\Phi(X_0^s X_1^{l-s})\delta_0^s\delta_1^{l-s}.
$$
\end{re}

From now on we will put $l\edf l_2$ to save on notations and  will thus use the  isomorphism $D_l$ to identify the space ${\cal P}\edf \P(\C[X_0,X_1]_{l}^*)$ to the projective space $\P(\C[\delta_0,\delta_1]_l)$, which can be obviously identified with the symmetric power $S^l(\P^1_\C)$ via the map
$$
P(\delta_0,\delta_1)\mapsto Z(P)
$$
which assigns to a differential operator  $P(\delta_0,\delta_1)$, the zero locus of $P$.
Therefore, from now on we will  identify ${\cal P}$ to  $S^l(\P^1_\C)$ via this map.

Our next goal: describe explicitly the fibres of the map $\Coker(H)\to {\cal P}\simeq S^l(\P^1_\C)$. We will make use of the  duality Theorem \ref{duality-th-intro} stated in the introduction, which gives, for a fixed linear form $\Phi\in  \C[X_0,X_1]^*_{l}$, a precise interpretation of the associated linear map
$$
\C[X_0,X_1]_d\ni Q\mapsto \trp{m}_Q(\Phi)\in \C[X_0,X_1]^*_{l-d}
$$
whose cokernel $C_\Phi^d\edf \mathrm{coker} (Q\mapsto \trp{m}_Q(\Phi))$ is the space 
we are interested in. We give below the proof of this theorem:
%
%

\begin{proof} (of the duality Theorem \ref{duality-th-intro}):  We have to prove that for any $P(\delta_0, \delta_1)\in \C[\delta_0,\delta_1]_l$ we have equality 
\begin{equation}
\trp{m}_Q\big(   D_l(P(\delta_0,\delta_1))\big)=D_{l-d}\bigg(Q\left(\frac{\partial}{\partial \delta_0}, \frac{\partial}{\partial \delta_1}\right)(P(\delta_0, \delta_1)) \bigg)
\end {equation}
in $\C[X_0,X_1]_{l-d}^*$. It suffices to check the equality for monomials. Choose  $P(\delta_0,\delta_1)=\delta_0^{s_0}\delta_1^{s_1}$,  $Q(X_0,X_1)=X_0^{k_0}X_1^{k_1}$ with $s_0+s_1=l$, $k_0+k_1=d$. For any monomial $X_0^{a_0} X_1^{a_1}\in \C[X_0,X_1]_{l-d}$ we have

\begin{align}
\langle \trp{m}_Q\big(   D_l(P(\delta_0,&\delta_1))\big),X_0^{a_0} X_1^{a_1}\rangle=\langle D_l(P(\delta_0,\delta_1)),X_0^{a_0+k_0}X_1^{a_1+k_1}\rangle =\nonumber\\
&=\left\{
\begin{array}{ccc}
(a_0+k_0)!(a_1+k_1)! & \rm if & (a_0+k_0,a_1+k_1)=(s_0,s_1)\phantom{,}	\\
0  & \rm if & (a_0+k_0,a_1+k_1)\ne (s_0,s_1),
\end{array}
\right. \label{left-term}
\end{align}	

and

\begin{align}
\bigg \langle D_{l-d}\bigg(Q&\left(
\frac{\partial}{\partial \delta_0},  \frac{\partial}{\partial \delta_1}
\right)
(P(\delta_0, \delta_1)) \bigg),X_0^{a_0} X_1^{a_1}\bigg \rangle=\bigg\langle 
D_{l-d}\bigg(\frac{\partial^d(\delta_0^{s_0}\delta_1^{s_1})}
{\partial \delta_0^{k_0}\partial\delta_1^{k_1}} \bigg),X_0^{a_0} X_1^{a_1}\bigg \rangle \nonumber\\&=
\frac{s_0!}{(s_0-k_0)!}\frac{s_1!}{(s_1-k_1)!}\frac{}{}\big\langle 
D_{l-d} ( \delta_0^{s_0-k_0}\delta_1^{s_1-k_1} ),X_0^{a_0} X_1^{a_1}\big \rangle\nonumber\\ 
\label{right-term}
&=\left\{
\begin{array}{ccc}
s_0! s_1! &\rm if &(s_0-k_0,s_1-k_1)=(a_0,a_1)\phantom{.}\\
0&\rm if &(s_0-k_0,s_1-k_1)\ne (a_0,a_1)	.
\end{array}\right.
\end{align}

The claim follows after comparing (\ref{left-term}) with (\ref{right-term}).
\end{proof}

We will make use of the duality theorem Theorem \ref{duality-th-intro} is the following way: we will replace the variables $\delta_0$, $\delta_1$ in $P(\delta_0, \delta_1) $ by the standard variables $X_0$, $X_1$ used in the theory of polynomials and we will replace the differential operators $\frac{\partial}{\partial \delta_0}, \frac{\partial}{\partial \delta_1}$ intervening in  $Q\left(\frac{\partial}{\partial \delta_0}, \frac{\partial}{\partial \delta_1}\right)$ by the partial differentials  $\frac{\partial}{\partial X_0}$, $\frac{\partial}{\partial X_1}$.  In this way we will be able to make use of the theory of apolar ideals explained in section \ref{apolar-section}.


Let $\Phi \in \C[X_0,X_1]_l^*\setminus\{0\}$ be a non-trivial linear form and  let $P\in \C[X_0,X_1]_l\setminus\{0\}$ be the polynomial associated with $\Phi$   using the identifications explained above. The cokernel $C_\Phi^d$ in which are interested is obviously identified with
$$
C_P^d=\coker\big (\C[\xi_0,\xi_1]_d\ni Q\mapsto  Q\bullet P\in \C[X_0,X_1]_{l-d}\big).
$$

As explained above, the goal of this section is the computation of the dimensions $\dim(C_P^d)$ as $[P]$ varies in the projective space $\P(\C[X_0,X_1]_{l})$.  We will answer this problem in two ways:

\begin{itemize}
\item We will give an explicit formula for the value of  $\dim(C_P^d)$ on each stratum $P^l_r$ of the cactus rank stratification of 	$\P(\C[X_0,X_1]_{l})$. 

\item We will express the level sets $P^l_{d,s}$ ($0\leq s\leq  l-d+1$) of the map 
$$\P(\C[X_0,X_1]_{l})\ni [P] \mapsto \dim(C_P^d)\in\N $$
in terms of the strata of the  cactus rank stratification of $\P(\C[X_0,X_1]_{l})$ (see section \ref{apolar-section}).
\end{itemize}
\vspace{2mm}

The definition of $C_P^d$ gives:
\begin{re}
Under the assumptions and with the notations above we have 	
$$
\dim(C_P^d)=l-2d+\dim(\Ann(P)_d) 
$$
for any $0\leq d\leq l$. In particular $\dim(C_P^d)\geq l-2d$.
\end{re}

 By Theorem \ref{SlvThr} explained in the Appendix, the apolar ideal $\Ann(P)_d$ is generated by two relative prime  homogeneous polynomials $G_1$, $G_2$ satisfying $\deg(G_1)\leqslant \deg(G_2)$, $\deg(G_1)+ \deg(G_2)=l+2$.
 
The dimension $\dim(\Ann(P)_d)$ can be easily computed in terms of $d_i(P)\edf \deg(G_i)$, see   Remark \ref{dim(Ann(P))Rem} proved in the Appendix. We obtain:

\begin{pr}\label{dim-prop} Let  $P\in\C[X_0,X_1]_l\setminus\{0\}$ and $0\leq d\leq l$.  

\begin{enumerate} 
\item We have 
\begin{equation}\label{dim-prop-formula}
\dim(C_P^d)= \left\{
    \begin{array}{ll}
        0 & \mbox{if } l+2-\crk(P)\leqslant d \leqslant l, \\
        l-d+1-\crk(P) & \mbox{if } \crk(P) \leqslant d < l+2-\crk(P), \\
        l-2d & \mbox{if } d < \crk(P).
    \end{array}
\right.
 \end{equation}

 \item The dimension $s:=\dim(C^d_P)$ depends only on $d$ and the cactus rank $r=\crk(P)$:

\begin{enumerate}
\item For $0\leq d\leq [\frac{l}{2}]$, the function $s=s(r)$ is given by:
\vspace{2mm}
\\
\begin{overpic}[width=0.7\textwidth]{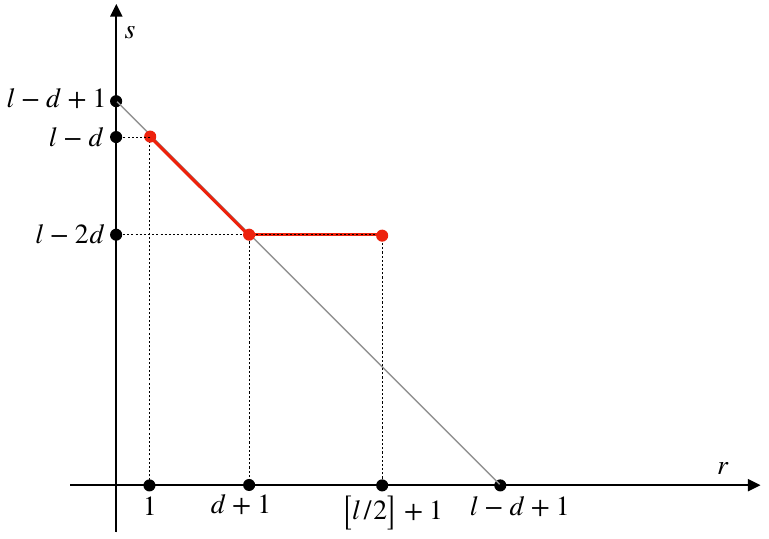}
 \put (30,60) {$s(r)=\left\{
\begin{array}{lcl}
l-d-(r-1)&\rm if & 1\leq r\leq d+1,\\
l-2d&\rm if &d+1\leq r\leq [\frac{l}{2}]+1.
\end{array}\right.$}
 \end{overpic}

\item For $[\frac{l}{2}]\leq d\leq l$, the function $s=s(r)$ is given by: 
\vspace{2mm}
\\
\begin{overpic}[width=0.7\textwidth]{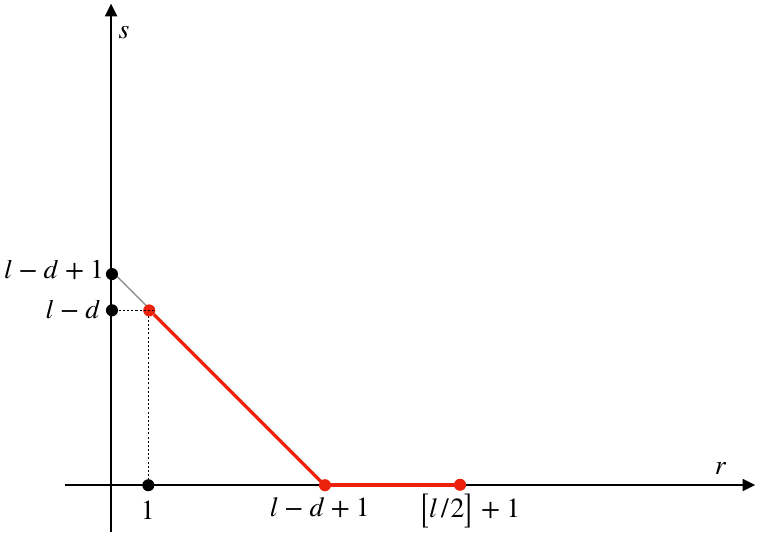}
 \put (30,60) {$s(r)=\left\{
\begin{array}{lcl}
l-d-(r-1)&\rm if & 1\leq r\leq l-d+1,\\
0&\rm if &l-d+1\leq r\leq [\frac{l}{2}]+1.
\end{array}\right.$}	
\end{overpic}

%
\end{enumerate}
\end{enumerate}
 \end{pr}

We will now give explicit formulae for the level sets of the map $[P]\mapsto \dim(C_P^d)$ in terms of the strata of the cactus rank stratification of $\P(\C[X_0,X_1]_l)$:

For $1\leq r\leq \big[\frac{l+2}{2}]$ consider the locally closed partition $(P^l_r)_{1\leq r\leq [\frac{l+2}{2}] }$  of the projective space  $\P(\C[X_0,X_1]_l)$, where
$$
P^l_r\edf \{[P]\in\P(\C[X_0,X_1]_l)|\ \crk(P)=r\},
$$
(see Remark \ref{partition-rem} in the Appendix). For $0\leq d\leq l$ and $0\leq s\leq  l-d+1$  define
$$P^l_{d,s}\edf \{[P]\in\P(\C[X_0,X_1]_l)|\ \dim(C_P^d)=s\}$$

Using Proposition \ref{dim-prop}, we obtain

\begin{pr} Let $d$, $s\in\N$ with $0\leq d\leq l$ and $0\leq s\leq l-d+1$. We have:    
		\begin{equation}\label{Plds}
\begin{split}
P^l_{d,s}=&\left\{\begin{array}{lcl}
\union_{r>d} P^l_r &\rm if & s=l-2d,\\  
P^l_{l-d+1-s} &\rm if &s>l-2d\hbox{ and } s> 0,	\\  
\bigcup_{r\geq l-d+1} P^l_r&\rm if &s>l-2d\hbox{ and } s= 0. 	
\end{array}
\right.\
\end{split} 
\end{equation}
\end{pr}

\begin{proof}

1) The case   $s=l-2d$.  The inclusion $P^l_{d,l-2d}\supset \union_{r>d} P^l_r$ is obvious by the third formula in (\ref{dim-prop-formula}).  For the opposite inclusion, let $[P]\in P^l_{d,s}$ with $s=l-2d$; suppose that $r=\crk(P)\leq d$.  It follows $l+2-\crk(P)-d=d+s+2-\crk(P)>0$, so the two inequalities defining the second case in  (\ref{dim-prop-formula}), so $l-2d=l-d+1-\crk(P)$, so $\crk(P)=d+1$, which contradicts the assumption $\crk(P)\leq d$.
\vspace{2mm}\\
2) The case   $s>l-2d$ and $s>0$. 

In this case we use the second formula in (\ref{dim-prop-formula}) and we see that we must have
$$
s= l-d+1-\crk(P) \hbox{ with } \crk(P) \leqslant d < l+2-\crk(P).
$$
We obtain $\crk(P)=l-d+1-s$ and the condition
$$
l-d+1-s\leq d< l+2-(l-d+1-s)=d+1+s
$$
is obviously satisfied. 
\vspace{2mm}\\
3) The case $s>l-2d$ and $s= 0$.

In this case $l-2d<0$, so the third formula (\ref{dim-prop-formula}) cannot yield $0$. Taking into account the first two formulae in (\ref{dim-prop-formula})  we obtain the result $s=0$ when 

\begin{enumerate}
	\item[(a)] $l+2-\crk(P)\leqslant d \leqslant l$, or
	\item[(b)] $l-d+1-\crk(P)=0$ and $\crk(P) \leqslant d < l+2-\crk(P)$
\end{enumerate}

The case (a) gives $\crk(P)\geq l+2-d$ and the condition $d\leq l$ is satisfied.  The case (b) gives $\crk(P)=l-d+1$ and the condition $\crk(P) \leqslant d < l+2-\crk(P)$ is satisfied. Therefore cases (a), (b) give  the solutions $\crk(P)\geq l-d+1$.   
\end{proof}

\begin{co}\label{stratif-comparison}

	We have
	
	\begin{enumerate}
\item For $0\leq d\leq [\frac{l}{2}]$, we have
$$
P^l_{d,s}=\left\{
\begin{array}{ccc}
P^l_{l-d+1-s}&\rm if & l-2d< s\leq l-d,\\
\union_{d+1\leq r\leq [\frac{l}{2}]+1} P^l_r &\rm if & s=l-2d,\\
\emptyset  &\rm if & s< l-2d \hbox{ or } s> l-d.
\end{array}\right.	
$$
\item For $[\frac{l}{2}]\leq d\leq l$, we have
$$
P^l_{d,s}=\left\{
\begin{array}{ccc}
P^l_{l-d+1-s}&\rm if & 0< s\leq l-d,\\
\union_{l-d+1\leq r\leq [\frac{l}{2}]+1} P^l_r &\rm if & s=0,\\
\emptyset  &\rm if &    s> l-d.
\end{array}\right.	
$$
\end{enumerate}
 
\end{co}

\section{Appendix}

\subsection{The cokernel problem for linear spaces} \label{coker-appendix}

Let $X$ be a complex manifold. Recall \cite{Fi} that category of linear spaces over $X$ is anti-equivalent to the category of coherent sheaves over $X$ under the functor which associates to a linear space $L$ on $X$ the coherent sheaf $\mathscr{L}(L)$ of linear forms on $L$.

Using this equivalence one can define
\begin{dt}
Let $\varphi:L\to L'$ be a morphism of linear spaces over $X$ and let $\varphi^\smvee: \mathscr{L}(L')\to \mathscr{L}(L) $	be the associated morphism of coherent sheaves. The categorical cokernel of $\varphi$ is the linear space $\coker(\varphi)$ associated with the coherent sheaf $\ker(\varphi^\smvee)$.
\end{dt}

Note that $\coker(\varphi)$ comes with a natural epimorphism of linear spaces 
$$
c:L'\to \coker(\varphi)
$$ 
with $c\circ \varphi=0$  satisfying the following universal property: for any linear space morphism $\gamma:L'\to L''$ with $\gamma\circ \varphi=0$, there exists an unique $\tilde \gamma:\coker(\varphi)\to L''$ such that $\gamma=\tilde\gamma\circ c$.

\begin{re}
In general
\begin{enumerate}
\item An epimorphism of linear spaces is not necessarily  fibrewise surjective.  
\item The underlying total space of the cokernel $\coker(\varphi)$ of a morphism of linear spaces over $X$ does not coincide with the union $\cup_{x\in X} \coker(\varphi_x)$ of fibrewise cokernels of $\varphi$. 	
\end{enumerate}

\end{re}

\begin{ex}
Let $L$ be a line bundle over $X$ with $H^0(X,{\cal L})\ne 0$ and let $f: {\cal O}_X\to {\cal L}$ be a non-trivial section of $L$ (or, equivalently, of the associated invertible sheaf ${\cal L}$). Note that $\varphi$ defines a morphism of (locally trivial) linear spaces 
$$
 \varphi:X\times\C\to L
$$
over $X$. The morphism of coherent sheaves which corresponds to $  \varphi$ via the anti-equivalence $\mathscr{L}$  is $\trp{f}: {\cal L}^*\to {\cal O}_X^*={\cal O}_X$.

Denoting by $D$ the zero devisor of $\varphi$, we can identify this morphism with the canonical embedding ${\cal O}_X\hookrightarrow  {\cal O}_X(D)$, and, via this identification, $\trp{f}$ coincides with the canonical embedding ${\cal O}_X(-D)\hookrightarrow {\cal O}_X$. Therefore $\ker(\trp{f})=0$, so $\coker(\varphi)=0$. This shows that $\varphi$ is always an epimorphism of linear spaces. But, when $D\ne\emptyset$, $\varphi$ is not fibrewise surjective.

\end{ex}

This justifies the following

\begin{dt}\label{DefCoker}
Let $p:L\to X$,  $p':L'\to X$ linear spaces over $X$ and
$\varphi:L\to L'$ be a morphism of linear spaces over $X$. The topological cokernel of $\varphi$ is the topological quotient of $L'$ by the equivalence relation $R_\varphi$  given by
$$
y'\mathrel{R_\varphi} z'  \hbox{ if }p'(y')=p'(z) \hbox{ and }\exists y\in L \hbox { such that } z'-y'=\varphi(y).
$$ 
This quotient will be denote by $\Coker(\varphi)$.		
\end{dt}

In other words, $\Coker(\varphi)$ is the underlying space of the relative quotient stack $L'/L$ over $X$, where $L$ acts on $L'$ via $\varphi$ fibrewise over $X$. This relative quotient stack is called the {\it stacky cokernel} of $\varphi$  in \cite[section I.3]{Pl2}.

\begin{ex}\label{non-Hausdorff-Coker} Consider the line bundles $L=L'=\C\times \C$ over $\C$ and the morphism $\varphi:L\to L'$ defined by $\varphi(z,\zeta)=z\zeta$. The topological cokernel $\Coker(\varphi)$ is a ``line with an infinite family of origins". The  fibre  $\Coker(\varphi)_z$ over  a point $z\in\C^*$ is trivial, whereas $\Coker(\varphi)_0$ is a complex line. Any two points in $\Coker(\varphi)_0$ are non-separable by disjoint neighbourhoods, so $\Coker(\varphi)$ is non-Hausdorff. 
\end{ex}

\begin{ex}
Suppose that $L$, $L'$ are holomorphic vector bundles over $X$, let ${\cal L}$, ${\cal L}'$ be the associated locally free shaves of locally defined sections and let $f:{\cal L}\to {\cal L}'$ be the sheaf morphism  associated with $\varphi$. The underlying set of   $\Coker(\varphi)$ can be identified with the disjoint union
$$
\coprod_{x\in X} \coker(f)(x) 
$$ 
of the fibres 	$\coker(f)(x)=\coker(f)_{x}/\mg_x\coker(f)_{x}$ of the coherent sheaf $\coker(f)$.  Taking into account a classical semi-continuity theorem (see \cite[section 4.4]{GrRe}) it follows that the map 
$$X\ni x\mapsto \dim(\coker(f)(x))=\dim(\Coker(\varphi)_x) $$
is upper semicontinuous with respect to the Zariski topology. 
\end{ex}

%
%
	
%

 \subsection{Apolar ideals and the Sylvester theorem. Cactus rank and the associated stratification}\label{apolar-section}
 
 \subsubsection{General Theory}

  Consider the apolar action 
$$\bullet  : \C[\xi_0,\xi_1] \times \C[X_0,X_1] \to \C[X_0,X_1],\ (Q,P)\mapsto Q \bullet P \edf Q(\frac{\partial}{\partial X_0},\frac{\partial}{\partial X_1})(P(X_0,X_1)).$$

For a fixed homogeneous polynomial $P\in \C[X_0,X_1]$ denote by $\Ann(P)$ the \textit{apolar ideal} of $P$ defined as below: 
$$
\Ann(P)\edf \{Q\in \C[\xi_0,\xi_1]| \ Q\bullet P=0\ \}.
$$

\begin{ex} \label{xmym} We have
	$$\Ann(X_0^m X_1^m (X_0+X_1))=(\xi_0^{m+1}-\sum_{k=1}^m (-1)^k \xi_0^{m-k}\xi_1^k, \xi_1^{m+2}).$$
This formula is obtained by prooving that the remainder of the Euclidian division of a polynomial $g\in \Ann(X_0^m X_1^m (X_0+X_1))$ by $\xi_1^{m+2}$ is divisible by $\xi_0^{m+1}-\sum_{k=1}^m (-1)^k \xi_0^{m-k}\xi_1^k$.
\end{ex}
We denote by $\sim$  the   equivalence relation on $\C[X_0,X_1]_l$ induced by the right $\GL(2,\C)$-action induced by composition with linear transformations. %
One has the well-known result  (see \cite[Lemma 2.2]{DS2}):

\begin{lm}\label{Lemma-Dimca}

\begin{enumerate}
	\item $\Ann(P)_0 \neq 0$ if and only if $P=0$. 
	\item $\Ann(P)_0=0$ and $\Ann(P)_1\neq 0$ if and only if $P\sim X_0^l$. 
	\item $\Ann(P)_1 =0$ and $\Ann(P)_2=\C \lambda^2$ for a linear form $\lambda$ if and only if $P \sim X_0^{l-1}X_1$. 
	 
\end{enumerate}	
\end{lm}


Recall that the \textit{Waring rank} $\rk(P)$ of a binary form  $P\in \C[X_0,X_1]_l\setminus\{0\}$ is the smallest $r\in\N$ such that there exists linear forms $\lambda_1,\dots,\lambda_r$  with 
$$P=\lambda_1^l + \dots + \lambda_r^l. $$
Such a decomposition always exists because $\C[X_0,X_1]_l$ is generated by the powers $\lambda^l$ of linear forms on $\C^2$, see \cite[Theorem 4.2]{Re}, which also establishes the inequality $\rk(P)\leq  l$.

\begin{dt}\label{CactusRkDef} Let $P$ be a non-trivial binary form. The cactus rank $\crk(P)$ of $P$ is the minimal length $l(Z)$ of a $0$-dimensional subscheme $Z$ in $\P^1_\C$ defined by a homogeneous saturated ideal $I_Z$ contained in $\Ann(P)$. 
	
\end{dt} 

  Note that the Waring and the cactus ranks are obviously invariant under the natural $\GL(2,\C)$-action on $\C[X_0,X_1]_l$.

We have the following important \textit{structure theorem} due to Sylvester, see  \cite[Theorem 1.44, Claim p. 31]{IK},  \cite[Theorem 2.3]{DS2}: 

\begin{thr}\label{SlvThr} Let $P\in \C[X_0,X_1]_l\setminus\{0\}$.
	The ideal $\Ann(P)$ is generated by two homogeneous polynomials $G_1, G_2$ with no common zeros satisfying $\deg(G_1)+\deg(G_2)=l+2$ and $\deg(G_1) \leqslant \deg(G_2)$. 
	
	Moreover the Waring rank $\rk(P)$ is equal to $\deg(G_1)$ if $G_1$ has no multiple roots, and it is equal to $\deg(G_2)$ otherwise. 

\end{thr}

We will use the notation $d_i(P)\edf \deg(G_i)$ for $1\leq i\leq 2$. 

The dimensions of the homogeneous summands $\Ann(P)_d$ of the apolar ideal $\Ann(P)$ in the interesting case $d\leq \deg(P)$ can be easily computed in terms of the degrees $d_i(P)$:
 
\begin{re}\label{dim(Ann(P))Rem} Let $P$ be a non-trivial binary form. Then
$$
\dim(\Ann(P)_d)= \left\{
    \begin{array}{ll}
        2d-\deg(P) & \mbox{if } d_2(P)\leqslant d \leqslant \deg(P), \\
        d+1-d_1(P) & \mbox{if } d_1(P) \leqslant d < d_2(P), \\
        0 & \mbox{if } d < d_1(P).
    \end{array}
\right.
 $$

\end{re}
\begin{proof}

Any polynomial $\Pi\in \Ann(P)_d$ decomposes as 
$$
\Pi= \Pi_1 G_1+ \Pi_2 G_2
$$
where $\Pi_i\in \C[X_0,X_1]_{d-d_i(P)}$ (with the convention $\C[X_0,X_1]_s=0$ for $s<0$). This shows that the linear map
$$
\C[X_0,X_1]_{d-d_1(P)}\times \C[X_0,X_1]_{d-d_2(P)}\ni (\Pi_1,\Pi_2)\mapsto \Pi_1 G_1+\Pi_2 G_2
$$
is surjective.  The injectivity of this map is clear when $d< d_2(P)$, because in this case at least one of the two spaces $ \C[X_0,X_1]_{d-d_i(P)}$ vanishes. For  $d_2(P)\leq d\leq \deg(P)$ the injectivity follows taking into account that  two homogeneous polynomials $G_1$, $G_2$ have no commun zero and that $d_1(P)+d_2(P)=\deg(P)+2$.	
\end{proof}

Taking into account that $\deg(G_1)+\deg(G_2)=\deg(P)+2$, $\deg(G_1) \leqslant \deg(G_2)$ and that one of the two degrees $\deg(G_1)$, $\deg(G_2)$ coincides with the $\rk(P)$, it follows that (\cite[Corollary 2.5 ]{DS2}): 
\begin{re}\label{d1-versus-waring}
For any non-trivial binary form $P$ of degree $l$ we have
$$
d_1(P)=\left\{
\begin{array}{ccc}
\rk(P)&\rm if &\rk(P)\leq \frac{l+2}{2}\phantom{.}\\
l+2-\rk(P)&\rm if &	\rk(P)\geq \frac{l+2}{2}.
\end{array}
\right.
$$
	
\end{re}

\begin{re}\label{d1(P)=min(deg(H))} We have 
$$
d_1(P)=\min\{\deg(H)|\ H\hbox{ homogeneous, } H\in \Ann(P)\setminus\{0\}\}. 
$$
In other words 
$$d_1(P)=\min\{s\in\N^*|\   \C[\xi_0,\xi_1]_s\ni Q\mapsto Q\bullet P\in \C[X_0,X_1]_{l-s} \hbox{ is not injective}\}. $$
\end{re}

\begin{co}\label{d1(P)semic}
The map $\C[X_0,X_1]_{l}\setminus\{0\}\ni P\mapsto d_1(P)\in\N^*$ is lower semicontinuous. 	
\end{co}

Note that the Waring rank is in general {\it not} semicontinuous on $\C[X_0,X_1]_{l}\setminus\{0\}$. 

\begin{ex}\label{Wrk-comput}
\begin{enumerate}
	\item Let $m$, $n\in\N^*$. We have $\Ann(X_0^mX_1^n)=(\xi_0^{m+1},\xi_1^{n+1})$, so, by Theorem \ref{SlvThr},  $\rk(X_0^mX_1^n)=\max(n,m)+1$. 
	\item Let $P(X_0,X_1)=X_0^3X_1(X_0+X_1)$. Then $\Ann(P)=(X_1^3,G_2)$	where $G_2$ is a homogeneous polynomial of degree $4$. Then $\rk(P)=4$.
	\item Coming back to the Example \ref{xmym} we obtain $\rk(X_0^m X_1^m (X_0+X_1))=m+1$. 
\end{enumerate}
\end{ex}
The following theorem (see \cite[Theorem 2.4]{DS2}) concerns binary forms of $\rk(P) \geq  2$.

\begin{thr}\label{generic-rank}
	Let $P$ be a binary form of degree $l$ and Waring rank  $\rk(P) \geq  2$. The known inequality  $\rk(P)\leqslant \deg(P)$ becomes an equality if and only if $P\sim X_0^{l-1}X_1$. Moreover, for a generic binary form $P$ of degree $l$ we have 
	$$
	\rk(P)= \left [ \frac{l+2}{2} \right ].
	$$
\end{thr}

We are grateful to Alexandru Dimca for pointing me out the following result, which plays an important role in this paper:
 \begin{pr}\label{d1=cactus}
 Let 	$P\in \C[X_0,X_1]_{l}\setminus\{0\}$. Then $d_1(P)=\crk(P)$.
 \end{pr}
 \begin{proof}
 For a  homogeneous  ideal $I\subset \C[X_0,X_1]$ let $V(I)$ stand for subscheme of $\P^1_\C$ associated with $I$ \cite[section II.2]{Ha}. Note that the assignment $[G] \mapsto V(G)$ defines  bijection between $\P^1(\C[X_0,X_1]_d)$ and the set of $0$-dimensional subschemes    $V\subset \P^1_\C$ of length $d$. It suffices to note that the ideal $(G)$ generated by a  polynomial $G\in  \C[X_0,X_1]_{d}\setminus\{0\}$ is saturated and to take into account Remark \ref{d1(P)=min(deg(H))}.	
 \end{proof}
 Taking into account Proposition \ref {d1=cactus} and Remark \ref{d1-versus-waring} it follows:
 \begin{re}\label{crk-is-lower-semic}
 The map $\crk:	 \C[X_0,X_1]_{l}\setminus\{0\}\to \N^*$ given by the  cactus rank is lower semicontinuous. The cactus rank and the waring rank are related by the formula
 $$
\crk(P)=\left\{
\begin{array}{ccc}
\rk(P)&\rm if &\rk(P)\leq \frac{l+2}{2}\phantom{.}\\
l+2-\rk(P)&\rm if &	\rk(P)\geq \frac{l+2}{2}.
\end{array}
\right.
$$

 \end{re}

This implies:

 \begin{re}\label{crk=1}
Let $P$ be a non-trivial binary form of degree $l\geqslant 2$. Then $P\sim X_0^l$ if and only if $\crk(P)=1$.	\end{re}

Taking into account Theorem \ref{generic-rank}, we also infer:

\begin{re}\label{partition-rem} Put $P^l_r\edf \{[P]\in \P(\C[X_0,X_1]_l)|\ \crk(P)=r\}$. 

Then $(P_l^r)_{1\leq r\leq [\frac{l+2}{2}]}$ is a partition of the projective space $\P(\C[X_0,X_1]_l)$ in locally closed subsets. For $1\leq s\leq[\frac{l+2}{2}]$ the union  $\union_{1\leq r\leq s}P^l_r$ is Zariski closed and the union  $\union_{s\leq r\leq [\frac{l+2}{2}]}P^l_r$ is Zariski open.  The highest index stratum  $P^l_{[\frac{l+2}{2}]}$ is Zariski open and non-empty.   	
\end{re}

Note that 
$$
[P]\in P^l_{[\frac{l+2}{2}]}\Leftrightarrow 
\left\{
\begin{array}{lcc} 
\rk(P)=\frac{l+2}{2}&\rm if& l\in 2\N^*, \\
\rk(P)\in \left\{\frac{l+1}{2},\frac{l+3}{2}\right\}&\rm if& l\in 2\N+1.
\end{array}
\right.
$$
Remark \ref{partition-rem} shows that $(P_l^r)_{1\leq r\leq [\frac{l+2}{2}]}$ is a stratification of  $\P(\C[X_0,X_1]_l)$ (in the sense of  \cite[Definition 5.28.3, part 1]{Stacks}) with locally closed strata. This stratification will be called {\it the cactus rank stratification} of $\P(\C[X_0,X_1]_l)$.\\

There is a lower bound for the Waring rank of binary forms based in terms of their factorization to a product of linear forms (See \cite[Theorem 2.1]{Ne}). Let $P\in \C[X_0,X_1]_l\setminus\{0\}$ be a nonzero binary form of degree $l$ which   factorizes as
	$$
	P(X_0,X_1)\edf \prod_{i=0}^s \lambda_i(X_0,X_1)^{n_i},
	$$
	where $s\geqslant 1$, $\lambda_i$'s are pairwise linearly independent linear forms and $n_i\geqslant 1$. Then: 

\begin{equation}\label{rk-less-than-max-ni+1}
\rk(P)\geqslant \max(n_0,\dots,n_s)+1. 
\end{equation}

For example, let $P=\lambda_1^{d-2}\lambda_2\lambda_3$, where $d\geqslant 3$ and  the $\lambda_i$'s are pairwise linearly independent linear forms. Combining the inequality (\ref{rk-less-than-max-ni+1}) with \cite[Theorem 2.4]{DS2}, we obtain 
\begin{equation}\label{rk-lambda1{d-2}lamnda2lambda3}
\rk(\lambda_1^{d-2}\lambda_2\lambda_3)=d-1.	
\end{equation}

\subsubsection{The cactus rank stratification of \texorpdfstring{$\P(\C[X_0,X_1]_l)$}{C} for small \texorpdfstring{$l$}{l} }
\label{cactus-rk-strat-small-l}

In this section we will give explicit geometric descriptions of the  cactus rank stratification    of $\P(\C[X_0,X_1]_l)$ defined in the previous section.

We recall that the discriminant \cite[p. 23]{Mu} of binary forms defines a homogeneous map $\Delta: \C[X_0,X_1]_l\to \C$ of degree $2(l-1)$. A binary form $P$ has multiple roots if and only if $\Delta(P)=0$. \\
For $l\in\N^*$ we will denote by 
$$v_l: \P(\C[X_0,X_1]_1)\simeq \P^1_\C\to \P(\C[X_0,X_1]_l)$$
the Veronese map $[\lambda]\mapsto [\lambda^l]$.
\\
 
\paragraph{\it The case $l=2$.}
The level sets of the cactus rank  on the plane $\P(\C[X_0,X_1]_2)\simeq\P^2$ are:
$$
\crk^{-1}(1)=Z(\Delta)=\im(v_2),$$
$$ \crk^{-1}(2)=\P(\C[X_0,X_1]_2)\setminus Z(\Delta).
$$
Note that $\Delta(a_0X_0^2+ 2a_1 X_0X_1+a_2 X_1^2)=a_0 a_1-a_1^2$, so the closed stratum $$\crk^{-1}(1)=Z(\Delta)$$ is a conic. 
\\

\paragraph{\it The case $l=3$.} By Remark \ref{crk=1} the level sets of the cactus rank when on the projective space $\P(\C[X_0,X_1]_3)$  of binary cubic forms   are: 
$$
\crk^{-1}(1)=v_3(\P^1_\C),\ \crk^{-1}(2)=\P(\C[X_0,X_1]_3)\setminus \im(v_3).$$
\vspace{1mm}
\paragraph{\it The case $l=4$.} Recall that we have two $\GL(2,\C)$-invariant polynomials $g_2$ and $g_3$ on the space $\C[X_0,X_1]_4$ of binary quartic forms defined by the formula

$$\begin{vmatrix} a_0 &a_1 & a_2 +2 t \\ a_1 &  a_2-t &  a_3 \\  a_2+2t &  a_3 &  a_4 \end{vmatrix}=4t^3-g_2(Q)t+g_3(Q),$$
where
$$Q(X_0,X_1)=a_0 X_0^4+4a_1 X_0^3 X_1+6a_2 X_0^2 X_1^2 +4a_3 X_0 X_1^3 + a_4 X_1^4.$$
(see \cite[Proposition 1.25, p.27]{Mu}). The explicit formulae for the $g_i$ are
\begin{equation*}
\begin{split}
g_2(Q)&= a_0 a_4-4a_1 a_3+3a_2^2,\\
g_3(Q)&=a_0a_2a_4+2a_1a_2a_3 -a_2^3-a_0a_3^2-a_1^2a_4.	
\end{split}	
\end{equation*}
Moreover, by   \cite[Remark 1.26]{Mu}, we have the following identity
\begin{equation}\label{Delta-gi}
\Delta=g_2^3-27 g_3^2.	
\end{equation}

The $J$ invariant is defined as 
$$
J(Q)\edf \frac{g_2(Q)^3}{\Delta(Q)}$$
which is a regular function on  $\P(\C[X_0,X_1]_4)\setminus Z(\Delta)$. The fibres of $J$ are precisely the orbits of the  $\GL(2,\C)$-action on this   Zariski open set. 

We will use the following table which computes the Waring, the cactus rank and the value of $g_3$ for the canonical forms of binary quartics: 
 $$
 \begin{array}{|c|c|c|c|}
 \hline
 Q&\rk(Q) & \crk(Q)&g_3(Q)	\\ \hline
 X_0^4& 1& 1& 0 \\ \hline
 X_0^3X_1& 4 & 2 & 0\\ \hline
 X_0^2X_1^2& 3 & 3 & -\frac{1}{6^3} \\ \hline
 X_0^2X_1(X_0+X_1)& 3 & 3 & -\frac{1}{6^3}\\ \hline
 X_0X_1(X_0+X_1)(X_0+tX_1),\ t\in\{-1,\frac{1}{2},2\}& 2 & 2 & 0 \\ \hline
  X_0X_1(X_0+X_1)(X_0+tX_1),\ t\not\in\{-1,\frac{1}{2},2\} & 3 & 3 &\frac{(t-2)(t+1)(2t-1)}{144}\\ \hline
 \end{array}
 $$

For the Waring rank in the first three rows we have used Example \ref{Wrk-comput} and formula (\ref{rk-lambda1{d-2}lamnda2lambda3}). For the  Waring rank in the last two rows  we have used \cite[Theorem 3.1]{DS1} which stated that the Waring rank of a quartic form without multiple roots is either 3, when $J(Q)\ne 1$, or 2, when $J(Q)= 1$. By formula (\ref{Delta-gi}) we have $J(Q)= 1$ if and only if $g_3(Q)=0$. The cactus rank column has been filled in using Remark \ref{crk-is-lower-semic}; the values in the third column have been obtained by direct computation.

We can now prove
\begin{re}\label{crk-strat-l=4} The strata of the cactus rank stratification of the projective space $\P(\C[X_0,X_1]_4)$ of quartic forms are given by
$$
\crk^{-1}(1)=\im(v_4),
$$
$$
\crk^{-1}(2)=Z(g_3)\setminus \im(v_4),
$$
$$\crk^{-1}(3)=\P^3(\C) \setminus Z(g_3).
$$	
\end{re}
\begin{proof}
The first equality follows again by 	Remark \ref{crk=1}. The second and the third equalities follow by comparing the second and the third columns of the above table and taking into account the $\GL(2,\C)$-equivariance of $g_3$ and the $\GL(2,\C)$-invariance of the cactus rank.
\end{proof}

Therefore 

\begin{re} 
The filtration of $\P(\C[X_0,X_1]_4)$ with Zariski closed subsets associated with its cactus rank stratification  is
$$
\P(\C[X_0,X_1]_4)\supset   Z(g_3)\supset \im(v_4).
$$
 \end{re}

\end{document}